\documentclass[ptrf,11pt]{imsart}

\RequirePackage[OT1]{fontenc}
\usepackage{a4wide}
\usepackage{amsmath,amsthm,amssymb}
\usepackage[isolatin]{inputenc}
\usepackage{amsfonts,dsfont}
\usepackage{pstricks}
\usepackage{comment}
\usepackage{natbib}
\usepackage{colordvi}
\usepackage{graphicx}

\usepackage{xr-hyper}
\definecolor{vert}{rgb}{0.09,0.7,0.17}
\definecolor{violet}{rgb}{0.69,0.13,0.69}

\usepackage[pdftex,                %
bookmarks         = true,
bookmarksnumbered = true,
pdfpagemode       = None,
pdfstartview      = FitH,
pdfpagelayout     = SinglePage,
colorlinks        = true,
citecolor		= vert,
urlcolor          = violet,
pdfborder         = {0 0 0}
]{hyperref}

\startlocaldefs
\usepackage{amsthm}
\newtheorem{theorem}{Theorem}[section]
\newtheorem{proposition}[theorem]{Proposition}
\newtheorem{corollary}[theorem]{Corollary}
\newtheorem{lemma}[theorem]{Lemma}

\newtheorem{remark}[theorem]{Remark}

\newcommand{\R}{\mathbb{R}} 

\renewcommand{\P}{\mathbb{P}}
\newcommand{\E}{\mathbb{E}} 
\DeclareMathOperator{\var}{Var} 
\DeclareMathOperator{\cov}{Cov} 
\newcommand{\telque}{\;:\;}
\newcommand{\cond}{\;|\;}

\newcommand{\mtc}{\mathcal}
\newcommand{\mbf}{\mathbf}

\newcommand{\wt}[1]{{\widetilde{#1}}}

\newcommand{\ol}[1]{\overline{#1}}

\newcommand{\ind}[1]{{\mbf{1}\left\{#1\right\}}}

\newcommand{\eps}{\varepsilon}

\newcommand{\G}{{\widehat{\mathbb{F}}}_{m}}
\newcommand{\Gzero}{{\widehat{\mathbb{F}}}_{0,m}}
\newcommand{\Gone}{{\widehat{\mathbb{F}}}_{1,m}}
\newcommand{\Gchap}{{\widehat{\mathbb{G}}}_{m}}

\newcommand{\Gtilde}{{\widetilde{\mathbb{F}}}_{m}}

\usepackage{graphics}
\newcommand{\FDP}{\mbox{FDP}} 
\newcommand{\T}{\mathcal{T}}

\newcommand{\Z}{\mathbb{Z}}

\newcommand{\qnorm}{ {\Phi}^{-1} }
\newcommand{\pnorm}{ {\Phi}}

\renewcommand{\l}{\ell}
\endlocaldefs

\graphicspath{{../}}

\begin{document}

\begin{frontmatter}

\title{On empirical distribution function of high-dimensional Gaussian vector components with an application to multiple testing.}

\begin{aug}
\author{\fnms{Sylvain} \snm{Delattre}\ead[label=e2]{sylvain.delattre@univ-paris-diderot.fr}}
\address{
 Université Paris Diderot, LPMA,\\ 
\printead{e2}\\
\phantom{E-mail: sylvain.delattre@univ-paris-diderot.fr}}
\and
\author{\fnms{Etienne} \snm{Roquain}\ead[label=e1]{etienne.roquain@upmc.fr}}
\address{
UPMC Université Paris 6, LPMA,\\ 
\printead{e1}\\
\phantom{E-mail: etienne.roquain@upmc.fr}}
\runauthor{Delattre, S. and Roquain, E.}
\end{aug}

\begin{abstract}

This paper introduces a new framework to study the asymptotical behavior of the empirical distribution function (e.d.f.) of Gaussian vector components, whose correlation matrix $\Gamma^{(m)}$ is dimension-dependent.
Hence, by contrast with the existing literature, the vector is not assumed to be stationary. 
Rather, we make a ``vanishing second order" assumption ensuring that the covariance matrix $\Gamma^{(m)}$ is not too far from the identity matrix, while the behavior of the e.d.f. is affected by  $\Gamma^{(m)}$
only through the sequence $\gamma_m=m^{-2} \sum_{i\neq j} \Gamma_{i,j}^{(m)}$, as $m$ grows to infinity.
This result recovers some of the previous results for stationary long-range dependencies while it also applies to various, high-dimensional, non-stationary frameworks, for which the most correlated variables are not necessarily next to each other. 
Finally, we present an application of this work to the multiple testing problem, which was the initial statistical motivation for developing such a methodology. 
\end{abstract}

\begin{keyword}[class=AMS]
\kwd[Primary ]{60F17}
\kwd[; secondary ]{62G30}
\end{keyword}

\begin{keyword}
\kwd{empirical distribution function}\kwd{functional central limit theorem}\kwd{factor model}\kwd{Sample correlation matrix}\kwd{Gaussian triangular arrays}\kwd{Hermite polynomials}\kwd{functional delta method}  \kwd{false discovery rate} 
\end{keyword}


\end{frontmatter}


\section{Introduction}

\subsection{Motivation and background} 

Pertaining to the florishing field of statistics for high-dimensional data, the Benjamini-Hochberg (BH) procedure has become a well accepted and commonly used method when testing a large number of null hypotheses simultaneously. Its quality is measured via the false discovery proportion (FDP), the proportion of errors among the rejected null hypotheses, whose expectation is the celebrated false discovery rate (FDR), see \cite{BH1995}.
The methodology of \cite{Neu2008} shows that the FDP of BH procedure is an (Hadamard differentiable) functional of  empirical cumulative distribution functions (e.d.f. in short).  Via the functional delta method (see, e.g., \cite{Vaart1998}), this rises the problem of obtaining functional central limit theorems for e.d.f. in 
a setting which is suitable for high-dimensional data. 

A colossal number of work aimed at extending Donsker's theorem (\citealp{Doo1949,Don1952,Dud1966}) to a more relaxed setup.
  Among them, a particularly prospering research field deals with the introduction of weak dependence between the original variables, mainly by using mixing conditions.  Here, we do not attempt to provide an exhaustive list for such results and we refer the reader to, e.g.,  \cite{DP2007,DLST2010} for detailed reviews. 
When restricted to the Gaussian subordinated setting,  asymptotics for the e.d.f. are described in the two well-known papers of \cite{DT1989} (long-range) and \cite{CM1996} (short-range). Both studies make a \textit{stationarity} assumption: the covariance matrix between the variables is assumed to be of the form
$$
\Gamma_{i,j}^{(m)} = r(|i-j|), \:1\leq i,j\leq m,
$$
for some  function $r(\cdot)$ vanishing at infinity and not depending on $m$.

However, in high-dimensional data, while the dimension $m$ can be very large (typically, several thousands), the  matrix $\Gamma^{(m)}$  is generally complex and not-necessarily locally structured. This is typically the case when latent variables (factors) 
have a simultaneous impact on all the variables (see, e.g.,  \citealp{FKC2009,SZO2012, FHG2012} and references therein), which leads to ``spiked" correlation matrices (as refered to by \citealp{John2001}). 
In a more general view, the larger the dimension, the more stringent the stationary assumption.

\subsection{Presentation of the main result}\label{sec:setting}

Let us consider $\{Y^{(m)}, m\geq 1\}$ a triangular array for which each $Y^{(m)}=\left(Y^{(m)}_1, \dots, Y^{(m)}_m\right)$ is a $m$-dimensional Gaussian vector, defined on some probability space $(\Omega_m,\mtc{F}_m,\P_m)$, with zero mean and covariance matrix $\Gamma^{(m)}$. For the sake of simplicity, assume that each $Y_i^{(m)}$ is of variance $1$, that is, $\Gamma^{(m)}_{i,i}=1$ for all $i$.
Denote $\pnorm(z) = \P(Z\geq z)$, for $z\in \R$, $Z\sim\mathcal{N}(0,1)$, the upper tail distribution function of a standard Gaussian variable, 
and consider the empirical cumulative distribution function: 

\begin{equation}\label{equ-ecdf}
\G(t) = m^{-1} \sum_{i=1}^m  \ind{\pnorm(Y^{(m)}_i)\leq t},\:\:\:\mbox{ $t\in[0,1]$.}
\end{equation}
Here, we consider the e.d.f. of the $\pnorm(Y^{(m)}_i)$'s rather than the one of the $Y^{(m)}_i$'s to get uniformly distributed variables. 
The variables can therefore be interpreted as $p$-values, which is convenient for multiple testing, see Section~\ref{sec:appli}.
To study \eqref{equ-ecdf}, let us introduce the following quantities: 
  \begin{align}
  \gamma_m&=m^{-2}  \sum_{i\neq j} \Gamma_{i,j}^{(m)};\label{def:gammam}\\
  r_m&=\left(m^{-1} + \left| \gamma_m \right| \right)^{-1/2}.\label{rate}
  \end{align}
  In a nutshell, our main result is as follows: by assuming, when $m\rightarrow\infty$,
  \begin{align}
 \frac{r_m^2}{ m^{2}} \sum_{i\neq j} \left(\Gamma_{i,j}^{(m)}\right)^2 &\rightarrow 0\label{vanish-secondorder}\tag{\mbox{vanish-secondorder}};\\
 \frac{r_m^{4+\eps_0}}{m^2} \sum_{i\neq j} \left(\Gamma_{i,j}^{(m)}\right)^4 &\rightarrow 0, \:\:\:\mbox{ for some $\eps_0>0$;} \label{vanish-fourthorder}
 \tag{$H_1$}\\
 m \gamma_m &\rightarrow \theta ,\:\:\:\mbox{ for some $\theta\in [-1,+\infty]$;}\label{equtheta}\tag{$H_2$}
 \end{align}
the following weak convergence holds (in the Skorokhod topology): 
\begin{align}
r_m ( \G -  I)    \leadsto \Z  \label{aim}, \:\:\:\mbox{ as $m\rightarrow\infty$,}
  \end{align}
where $I(t)=t$ and $\Z$ is some continuous Gaussian process on $[0,1]$ with a distribution only function of $\theta$. Specifically, denoting $\phi$ the standard Gaussian 
 density,
\begin{itemize}
\item[(i)]  if $m\gamma_m \rightarrow \theta<+\infty$, we have $r_m\propto {m}^{1/2}$ and the process ${m}^{1/2}( \G - I)$ converges to a (continuous Gaussian) process with covariance function given by $(t,s)\mapsto t\wedge s -t s+ \theta \: \phi(\Phi^{-1}(t))\phi(\Phi^{-1}(s))$. Hence, the limit process is a standard Brownian bridge when $\theta=0$, but has a covariance function smaller (resp. larger)  if $\theta<0$ (resp. $\theta>0$).
\item[(ii)] if $m\gamma_m \rightarrow \theta=+\infty$, we have $r_m\sim (\gamma_m)^{-1/2} \ll {m}^{1/2}$ and $(\gamma_m)^{-1/2}( \G - I)$
converge to the process 
$\phi(\Phi^{-1}(\cdot)) Z$ for $Z\sim\mathcal{N}(0,1)$. Hence  the ``Brownian" part asymptotically disappears.
 \end{itemize}
The regimes (i) and (ii) are illustrated in Figure~\ref{fig:regimes}: as $m\gamma_m$ grows, the influence of the ``Brownian" part decreases while that of the  (randomly rescaled) function $\phi(\Phi^{-1}(\cdot))$ increases. Also, the scale of the $Y$-axis indicates that the ${m}^{1/2}$ is not a suitable rate for large values of $m\gamma_m$.

 \begin{figure}[h!]
\begin{center}
\begin{tabular}{cc}
$m\gamma_m=0$ & $m\gamma_m=2$\vspace{-.5cm}\\
\includegraphics[scale=0.35]{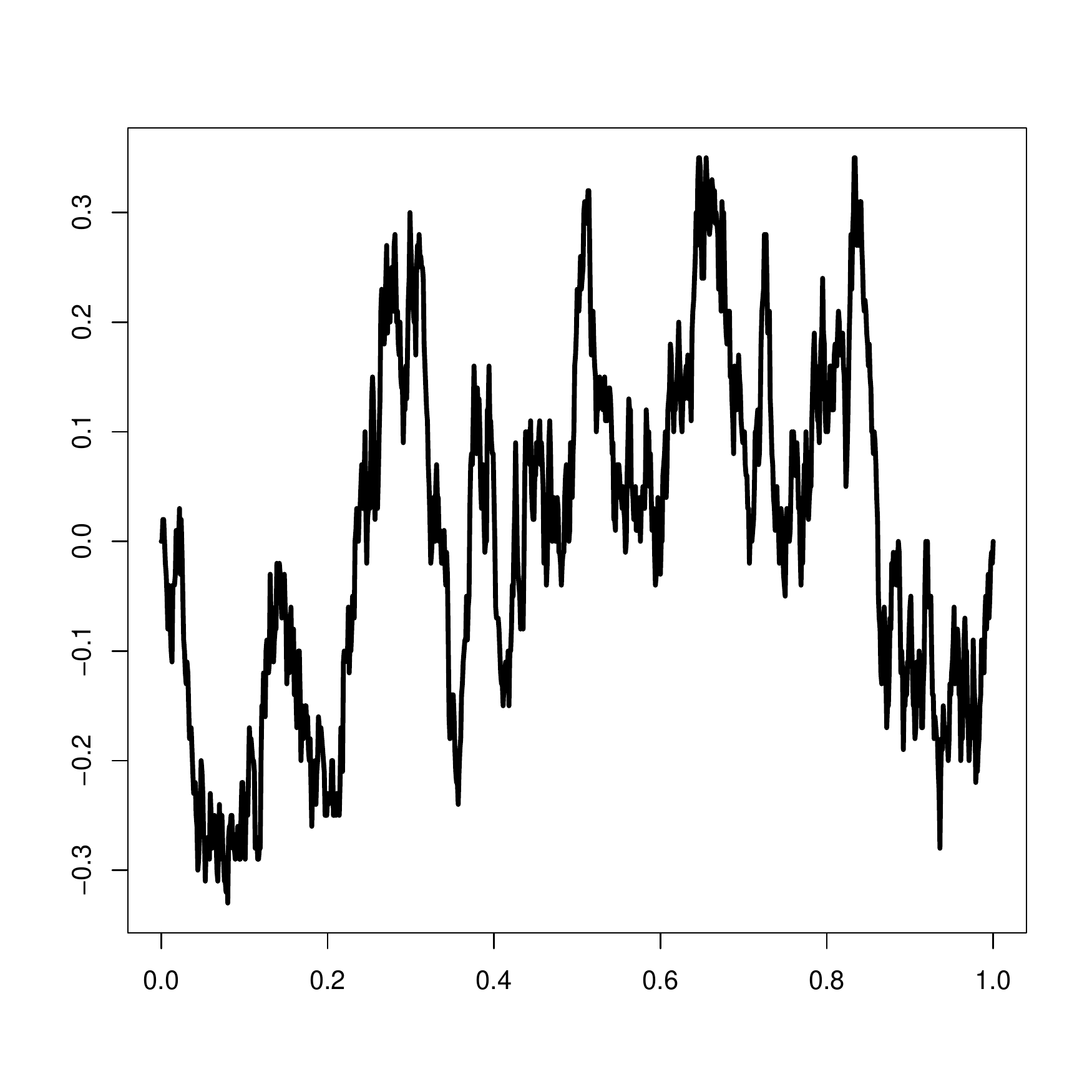} & \includegraphics[scale=0.35]{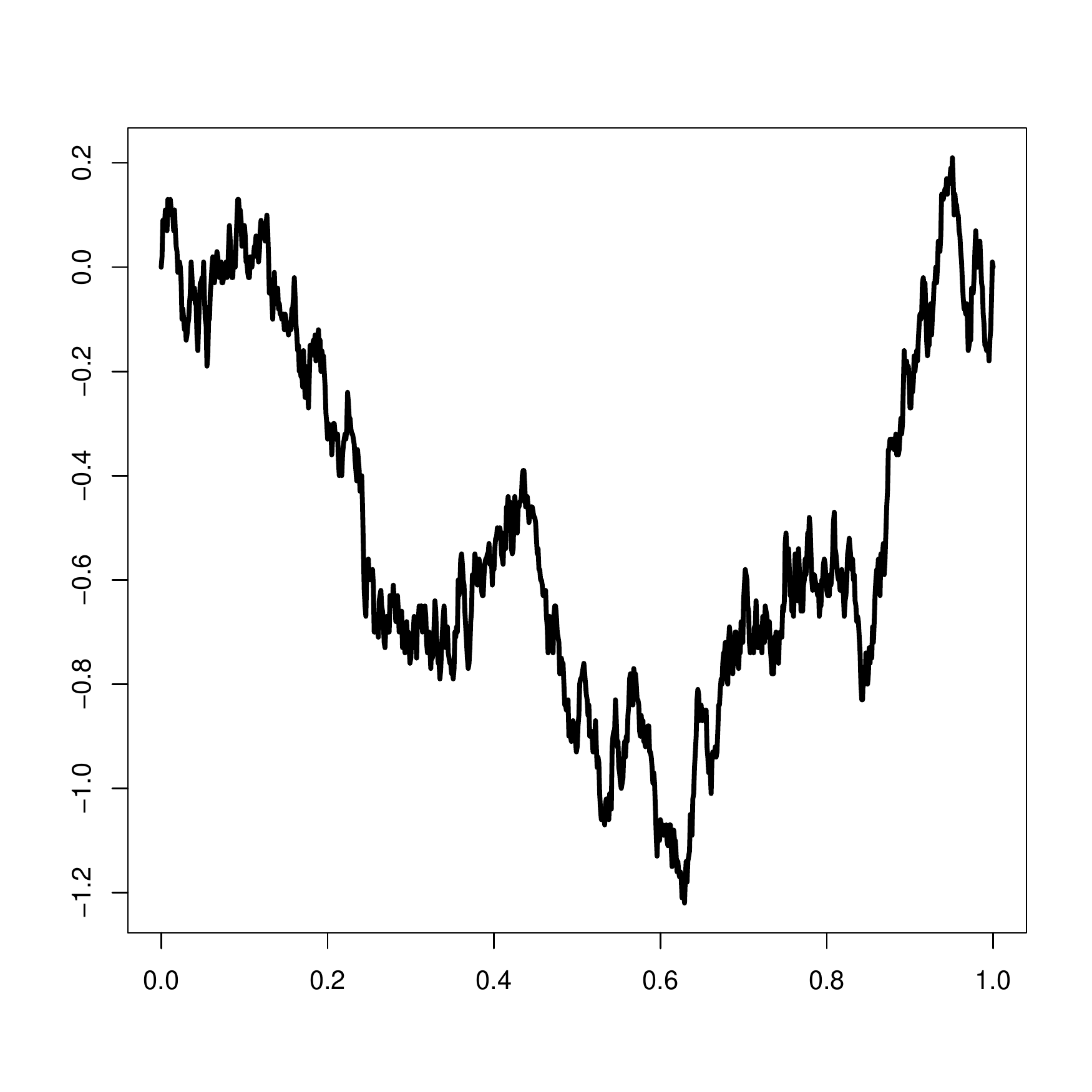}
\end{tabular}
\begin{tabular}{cc}
$m\gamma_m=10^2$&$m\gamma_m=10^3$\vspace{-.5cm}\\
\includegraphics[scale=0.35]{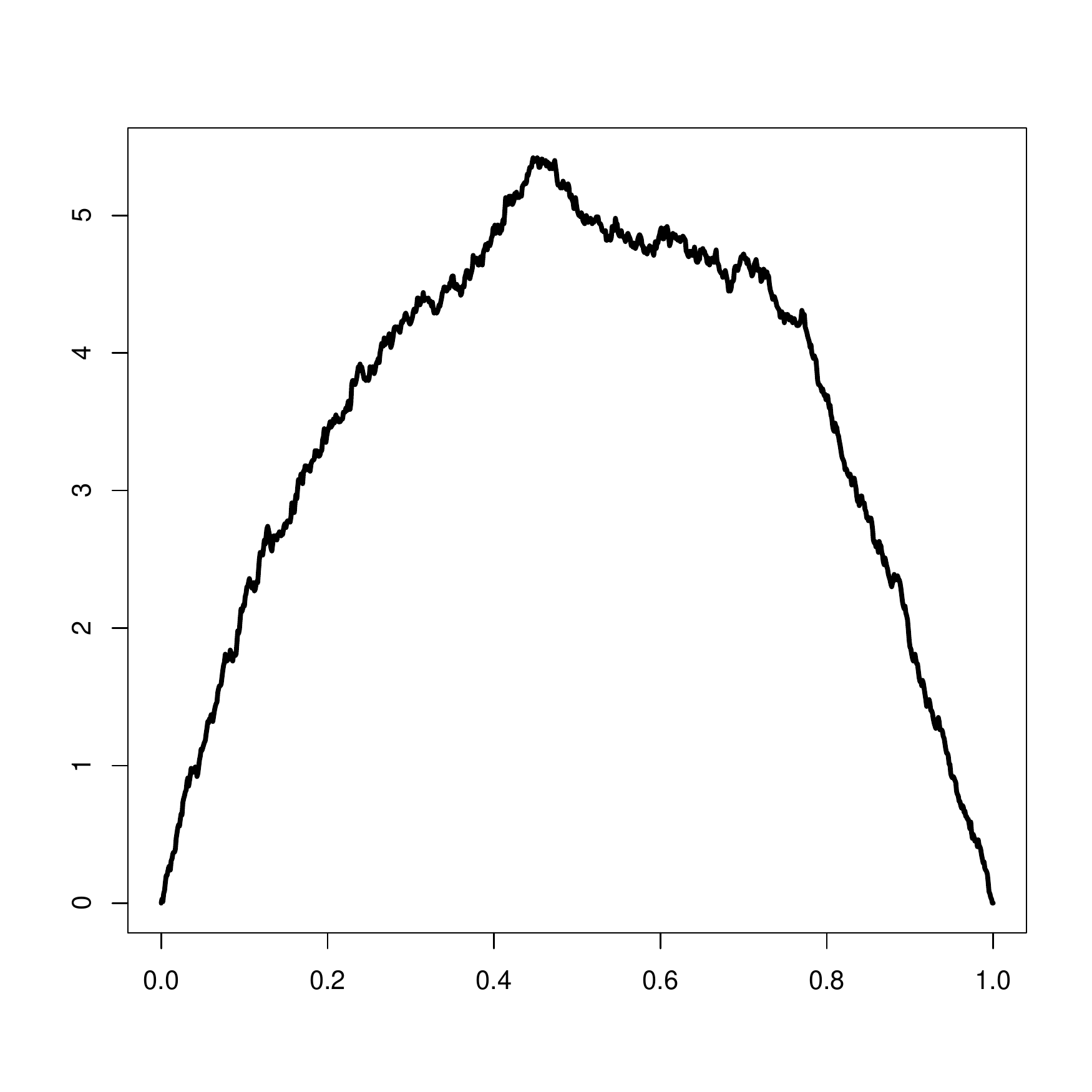}& \includegraphics[scale=0.35]{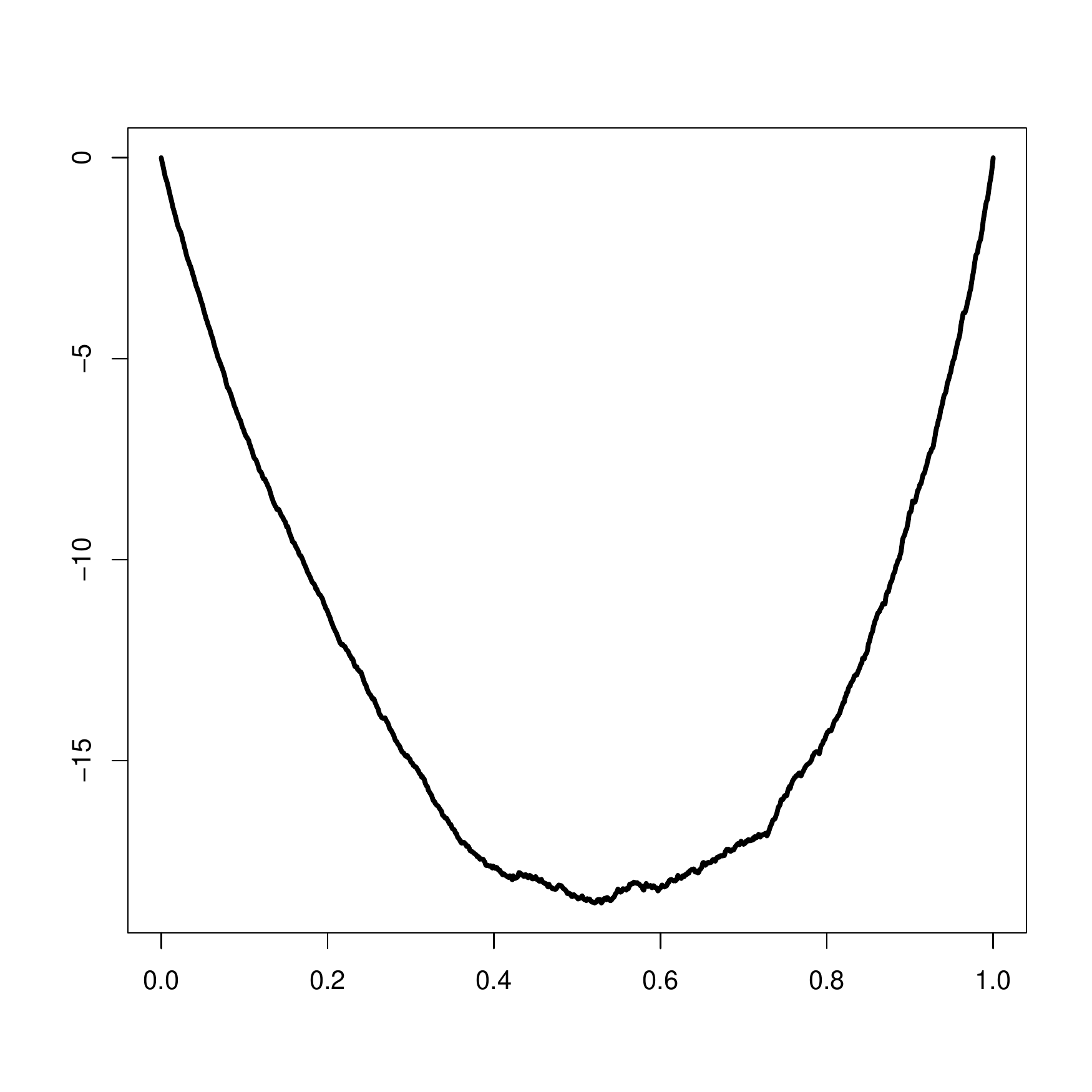}
\end{tabular}
\caption{Plot of $t\mapsto {m}^{1/2}( \G(t) - t)$ for some observed $Y(\omega)$. These realizations have been generated in the equi-correlated model $\Gamma^{(m)}_{i,j} =\rho_m$, $i\neq j$ (see \eqref{matrix:equi}) and for $m=10^4$.}
\label{fig:regimes}
\end{center}
\end{figure}

Let us briefly discuss our novel conditions. Condition~\eqref{vanish-secondorder} is the starting point of our study: it corresponds to assume that the expansion of the covariance function of $r_m ( \G - I)$ asymptotically stops at order $1$. This is a crucial $L^2$-type tool to elaborate our proofs in a possibly non-stationary regime. 
However, the price to pay is that it does not cover regimes where (some of) the greater orders matter asymptotically, as in the case of short range dependence (tridiagonal $1/2$-$1$-$1/2$ for instance).
As for Condition~\eqref{vanish-fourthorder}, it is only used to prove that $r_m ( \G - I)$ is $C$-tight and we suspect it to be unnecessary, athough we did not manage to remove it formally from our assumption set. 
Condition~\eqref{equtheta} is not restrictive because it holds up to consider a subsequence.

Finally,  we show that the convergence \eqref{aim} is maintained when replacing the set of assumptions \eqref{vanish-secondorder}, \eqref{vanish-fourthorder} and \eqref{equtheta} by the two following conditions:
\begin{align}
 \frac{r_m^{2+\eps_0}}{m^2} \sum_{i\neq j} \left(\Gamma_{i,j}^{(m)}\right)^2 =o(1), \:\:\:\mbox{ with $\eps_0>0$;} \label{vanish-secondorder+}
\tag{$H_3$} 
\\
 m\gamma_m^{1+\eps_0} \rightarrow +\infty, \:\:\:\mbox{ with $\eps_0>0$}\label{gammam+}\tag{$H_4$}.
 \end{align}
 Roughly speaking, it shows that, up to add some ``safety margin" $\eps_0$ in the convergence, Assumption~\eqref{vanish-fourthorder} can be removed in regime (ii).

  \subsection{Relation to existing literature} 
  
Compared to previous studies using the stationary paradigm, our assumptions are markedly different: first, the covariance matrix $\Gamma^{(m)}$ is allowed to depend on $m$, that is, the $Y^{(m)}$'s form a triangular array of Gaussian variables. 
Second, $\Gamma^{(m)}$ needs not be locally structured,
that is, $\Gamma^{(m)}_{i,j}$ is not necessarily related to the distance between $i$ and $j$.  
Instead, our conditions are permutation invariant, that is, are unchanged when permuting the columns of the triangular array. This is quite natural because the e.d.f. is itself permutation invariant. 
Third, 
our approach shows that the negative correlations can decrease the asymptotic covariance or even increase the convergence rate. 

As a counterpart, when restricted to the stationary setting, our assumptions are admittedly not optimal: 
it includes long-range of \cite{DT1989} but excludes short range of \cite{CM1996}. 
As explained above, this restriction comes from \eqref{vanish-secondorder}, which implicitly truncates the covariance expansion in the limit.

Nevertheless, our result opens a window for other dependence models as factor models or sample correlation matrices for instance. In particular, it covers the result of  \cite{DR2011}, obtained in the equi-correlated case where $\Gamma^{(m)}_{i,j} =\rho_m$, $i\neq j$, for some  correlation $\rho_m$ tending to zero (at some arbitrary rate).

Finally, let us mention the interesting work of \cite{BS2011} in which the stationarity assumption has also been removed, by establishing central limit theorems (CLT) for Gaussian subordinated arrays. There are two major differences with our approach: first, they deal with a CLT for the partial-sum process 
and not with a functional CLT for the e.d.f. Second, their assumptions are not of the same nature, because they require that $|\Gamma_{i,j}^{(m)}|\leq r(|i-j|)$ for all $i,j$, for some function $r(\cdot)$, independent of $m$, and vanishing at infinity.
    
  \subsection{Organization of the paper}

In Section~\ref{sec:convcov}, we study the covariance function of $\G$ under \eqref{vanish-secondorder}.
The main theorem is formally stated in Section~\ref{sec:mainresult} together with many illustrative examples. This new methodology is then applied to the multiple testing problem in Section~\ref{sec:appli}. 
 The proof of the main result is presented in Section~\ref{sec:proof}; it mainly relies on central limit theorems for martingale arrays and on a suitable  tightness criterion. To make the proof as clear as possible, some technical and auxiliary results are deferred to appendices.

\section{Preliminaries: covariance of $\G$ under \eqref{vanish-secondorder}.}\label{sec:prel}\label{sec:convcov}

Throughout the paper, to alleviate the notation, we will often denote $\P_m$ by $\P$, $Y^{(m)}$ by $Y$ and $\Gamma^{(m)}$ by $\Gamma$ when not ambiguous. 
  
Let us consider the sequence of Hermite polynomials $H_\l(x)$, $\l\geq 0$, $x\in \R$ (see Appendix~\ref{sec:Hermite}).  By using Melher's formula, the covariance function of the process $\G(\cdot)$ can be described as a function of the correlation matrix $\Gamma$ of $Y$. 

\begin{proposition}\label{prop:covG}
Consider $\G(\cdot)$ the process defined by \eqref{equ-ecdf} and the function family $\{c_\l(\cdot), \l\geq 1\}$ defined by 
 \begin{equation}\label{expan:clt} 
c_\l(t)=H_{\l-1}(\Phi^{-1}(t)) \phi(\Phi^{-1}(t)),\:\:\: \mbox{ $t\in[0,1]$, $\l = 1,2,\dots$.}
\end{equation}  
Then for all $t,s\in[0,1]$, we have
 \begin{align}
\cov(\G(t),\G(s))=&  \sum_{\l\geq 1} \frac{c_\l(t)c_{\l}(s)}{\l !} \left(m^{-2} \sum_{i,j}(\Gamma_{i,j})^\l\right).\label{equ:covG} 
\end{align} 
\end{proposition}
This result can be found, e.g., in Theorem~2 of \cite{SL2011} (see also Theorem~1 of \citealp{Efron2009}). We provide a proof in Appendix~\ref{sec:Hermite} for completeness.  
          While \eqref{equ:covG} is an exact expression, we can try to approximate the covariance $\cov(\G(t),\G(s))$ when $m$ grows to infinity, while making some assumption on the matrix $\Gamma=\Gamma^{(m)}$.

    Firstly, let us note the following: since $m^{-2} \sum_{i,j}(\Gamma_{i,j})^\l= (\l!)^{-1} \var\left( m^{-1} \sum_{i=1}^m H_\l(Y_i)\right)\geq 0$ (by using \eqref{prop:hermite} in  Appendix~\ref{sec:Hermite}), expression \eqref{equ:covG} shows that the following conditions are equivalent as $m$ tends to infinity,
   \begin{align}
& \forall t \in [0,1], \:\var(\G(t)) = o(1) \label{consistency} \\
&   \forall \l\geq 1, \: m^{-2} \sum_{i,j}(\Gamma_{i,j})^\l = o(1)\label{weakdepforall}\\
 &  m^{-2} \sum_{i,j}(\Gamma_{i,j})^2 = o(1)\label{weakdep}\tag{\mbox{LLN-dep}}
\end{align} 
   
As a consequence, Condition~\eqref{weakdep} is required as soon as a convergence result of the form \eqref{aim} holds. Note that the rate $r_m$ defined by \eqref{rate} satisfies $1\leq \left(m^{-2} \sum_{i,j} (\Gamma_{i,j})^2 \right)^{-1/4} \leq r_m \leq \sqrt{m}$. Hence $r_m$ tends to infinity under \eqref{weakdep} but not faster than $\sqrt{m}$.
 
Secondly, let us rewrite \eqref{equ:covG} as follows:
 \begin{align}
\cov(\G(t),\G(s))=& \:m^{-1} (t\wedge s -ts) +\gamma_m  c_1(t)c_1(s)\nonumber\\
&+ \sum_{\l\geq 2} \left(m^{-2}\sum_{i\neq j } (\Gamma_{i,j})^\l\right) c_\l(t)c_\l(s) (\l !)^{-1},\label{equ:covG-detail} 
\end{align} 
where $\gamma_m$ is defined by \eqref{def:gammam}.
The latter holds because, for two independent $\mathcal{N}(0,1)$ variables $U$ and $V$, we have $m^{-1} \sum_{\l\geq 1} c_\l(t)c_{\l}(s) (\l !)^{-1} = $ $ \cov(\ind{\Phi(U)\leq t}, \ind{\Phi(V)\leq s})$.
In expansion \eqref{equ:covG-detail}, the second order term (i.e., the sum over $\l\geq 2$) is negligible w.r.t. the other terms if \eqref{vanish-secondorder} holds.
Hence, assuming now \eqref{vanish-secondorder}, 
we obtain that the rescaled covariance  $\cov(r_m \G(t),r_m \G(s))$ of $r_m \G$ converges to the following covariance function 
\begin{align}
K(t,s)&= \frac{1}{1+|\theta|} (t\wedge s -t s)+ \frac{\theta}{1+|\theta|} c_1(t) c_1(s),  \label{equ:kernel}
\end{align}
where 
$\theta$ is defined in \eqref{equtheta} and where we use the conventions $\theta/(1+|\theta|)=1$ and $1/(1+|\theta|)=0$ when $\theta=+\infty$. 
Note that \eqref{equtheta}  always holds up to consider a subsequence, because $m\gamma_m \geq -1$ from the nonnegativeness of $\Gamma^{(m)}$.

\begin{remark} 
In the RHS of expression \eqref{equ:kernel}, the second term is not necessarily a covariance function because $\theta$ can be negative. Nevertheless, $K$ can be written as
$K(t,s)= \frac{1}{1+|\theta|} \wt{K}(t,s)  + \frac{1+\theta}{1+|\theta|} c_1(t) c_1(s)$,
where 
\begin{align}
\wt{K}(t,s)=t\wedge s -t s -c_1(t) c_1(s)\label{equKmin}
\end{align}
turns out to be a covariance function; considering a Wiener process $(W_t)_{t\in[0,1]}$, $\wt{K}$ is the covariance function of the process $W_t - t W_1 - c_1(t) \int_0^1 \Phi^{-1}(s) dW_s$, which is the orthogonal projection in $L^2$ of $W_t$ onto the orthogonal of the linear space spanned by $W_1$ and $\int_0^1 \Phi^{-1}(s) dW_s$. Interestingly, the latter also shows that the original covariance $K$ given by \eqref{equ:kernel} can be seen as the covariance function of 
$
\Z_t = (1+|\theta|)^{-1/2} \left( W_t - t W_1 \right) +   (1+|\theta|)^{-1/2}((1+\theta)^{1/2}-1) \:c_1(t) \int_0^1 \Phi^{-1}(s) dW_s.$
\end{remark}

  \section{Main result}\label{sec:mainresult}

\subsection{Statement}

Our 
 main result establishes that the convergence of the covariance functions investigated in Section~\ref{sec:convcov} can be extended to the case of a weak convergence of process. For this, we should consider
 the other technical assumptions described in Section~\ref{sec:setting}.

\begin{theorem}\label{main-thm}
Let us consider the empirical distribution function $\G$ defined by \eqref{equ-ecdf}. Assume that the covariance matrix $\Gamma^{(m)}$ depends on $m$ in such a way that \eqref{vanish-secondorder} and \eqref{vanish-fourthorder} hold with $r_m$ defined by \eqref{rate} and assume \eqref{equtheta}. 
Consider $(\Z_t)_{t\in[0,1]}$ a continuous process with covariance function $K$ defined by \eqref{equ:kernel}.
Then we have the convergence (in the Skorokhod topology) 
\begin{align}
r_m ( \G - I)    \leadsto \Z  , \mbox{ as $m\rightarrow\infty$,}\label{resmainthm}
  \end{align}
where $I(t)=t$ denotes the identity function.
Moreover, the result holds by replacing the set of assumptions $\{$\eqref{vanish-secondorder}, \eqref{vanish-fourthorder} and \eqref{equtheta}$\}$ by $\{$\eqref{vanish-secondorder+} and \eqref{gammam+}$\}$.
\end{theorem}

Theorem~\ref{main-thm} is illustrated in the next section, which provides several (commented) examples. 

\subsection{Examples}\label{sec:examples}

Let us first note that Assumptions \eqref{vanish-secondorder} and \eqref{vanish-fourthorder} always hold under the following condition
\begin{align}\label{assumeasy}
  \mbox{$|\Gamma^{(m)}_{i,j}|\leq a_m $ for all $i\neq j$ and $a_m$ satisfies $m^{1+\delta} a_m^2 \rightarrow 0$ for some $\delta>0$.}
\end{align}
Also remember that, as mentioned in Section~\ref{sec:setting}, regime (i) (resp. (ii)) referred to the case where $\theta<\infty$ (resp. $\theta=\infty$).
We  now give several types of matrix $\Gamma^{m}$ for which Theorem~\ref{main-thm} can be applied. 

 \paragraph{Equi-correlation}
 
Let us start with the following simple example:
\begin{align}\label{matrix:equi}
\Gamma^{(m)} = \left(\begin{array}{cccc} 1 & \rho_m & \dots & \rho_m\\\rho_m & \ddots & \ddots & \vdots \\ \vdots & \ddots & \ddots & \rho_m\\ \rho_m & \dots & \rho_m & 1\end{array}\right) = (1-\rho_m) I_m + \rho_m \left(\begin{array}{c} 1 \\\vdots \\  \vdots \\ 1\end{array}\right)\left(\begin{array}{c} 1 \\\vdots \\  \vdots \\ 1\end{array}\right)^T,
\end{align} 
where $\rho_m \in[ -(m-1)^{-1},1]$ is some parameter. 
We easily check that $\gamma_m$ defined by \eqref{def:gammam} is given by $m\gamma_m=(m-1) \rho_m$ and that the assumptions of Theorem~\ref{main-thm} are all satisfied if $\rho_m\rightarrow 0$ and $m\rho_m$ converges to some $\theta\in[-1,+\infty]$, 
which yields convergence \eqref{resmainthm}.
This is in accordance with Lemma~3.3 of \cite{DR2011}. 

This simple example already shows that, following the choice of the sequence $(\rho_m)_m$, 
the empirical distribution function can have various asymptotic behaviors. For instance, taking $\rho_m=-(m-1)^{-1}$ gives a  process in regime (i) with a minimal asymptotic covariance function ($\theta=-1$, see \eqref{equKmin}), while taking $\rho_m \sim m^{-2/3}$ leads to a rate $r_m\sim m^{1/3}\ll m^{1/2}$ and thus a process converging in regime (ii).
 
 \paragraph{Alternate equi-correlation}
 
 Let us consider the covariance matrix: 
\begin{align}
\Gamma^{(m)} = \left(\begin{array}{ccccc} 1 & -\rho_m & \rho_m & \dots & \\
-\rho_m & 1 & -\rho_m & \ddots& \vdots \\ 
\rho_m & \ddots & \ddots &\ddots& \rho_m \\ 
\vdots & \ddots & -\rho_m & 1& -\rho_m\\ 
 & \dots & \rho_m& -\rho_m & 1\end{array}\right) 
 = (1-\rho_m) I_m + \rho_m \left(\begin{array}{c} 1 \\ -1 \\  \vdots \\ 1\end{array}\right)\left(\begin{array}{c} 1 \\ -1 \\  \vdots \\ 1\end{array}\right)^T
,\label{matrix:alternate-equi}
\end{align}
where $\rho_m \in[ -(m-1)^{-1},1]$ is a given parameter.  Clearly, $\gamma_m$  is such that 
\begin{align*}
m\gamma_m&= 2 \rho_m m^{-1}\sum_{i=1}^{m-1} \sum_{k=1}^{m-i} (-1)^k =\rho_m  \lfloor m/2\rfloor/(m/2)
\end{align*}
Hence the rate $r_m$ defined by \eqref{rate} is $r_m\sim \sqrt{m}$ and assumptions of Theorem~\ref{main-thm}  are fulfilled (with $\theta=0$) by assuming that $m^{1+\delta} \rho_m^2 \rightarrow 0$, with $\delta>0$ (because \eqref{assumeasy} holds). Hence, under that assumption,  
$\sqrt{m}(\G-I)$ converges to a standard Brownian bridge.

Maybe surprisingly,
this example shows that, even if the correlations are ``strong" (e.g., $\rho_m \sim m^{-2/3}$, to be compared with the equi-correlated case), positive and negative correlations can exactly compensate each other to provide the same convergence result as under independence. 

\paragraph{Long-range stationary correlations}

Let us consider the  correlation matrix of the following form:  
\begin{equation}\Gamma^{(m)}_{i,j}=r(|j-i|),\:\:\: \mbox{ for } r(0)=1, \:\:r(k)=k^{-D} L(k),\:\:\: 0<D<1,\label{matrix:vanish-long-range}
\end{equation}
where $L:(0,+\infty)\rightarrow (0,+\infty)$ is slowly varying at infinity ($\forall t>0$, $ L(t x)\sim L(x)$ as $x\rightarrow +\infty$). This framework is often referred to as ``long-range dependence" in literature dealing with a stationary setup (see, e.g., \cite{DT1989,DLS2002}). First, standard calculations easily  show that for all $\nu\geq 0$,
  \begin{align}
m^{-1}\sum_{i\neq j}  |j-i|^{-\nu} = 2  m^{-1} \sum_{i=1}^{m-1} \sum_{k=1}^{i} k^{-\nu}  \sim \left\{ \begin{array}{cl} 2 \frac{ m^{1-\nu}}{(1-\nu)(2-\nu)} & \mbox{ if $\nu\in  [0,1)$}\\ 2\log m & \mbox{ if $\nu=1$}\\ 2\sum_{k\geq 1} k^{-\nu} & \mbox{ if $\nu>1$}\\\end{array} \right..\label{easycomput}
\end{align}
Thus, for any 
$\nu_1\in (D,1)$, since $L$ is slowly varying,
  \begin{align*}
m\gamma_m \gtrsim m^{-1}\sum_{i\neq j}  |j-i|^{-\nu_1} \gtrsim m^{1-\nu_1}   ,
\end{align*}
by applying \eqref{easycomput}, where the ``$u_m\lesssim v_m$" means $u_m=O(v_m)$. 
This entails $m \gamma_m^{1+(1-\nu_1)/(2\nu_1)} \gtrsim m^{(1-\nu_1)/2}$  and thus Assumption \eqref{gammam+} holds. In particular, $r_m\sim \gamma_m^{-1/2}$. Additionally, for any $\nu_2\in (D,1)$ and $\nu_3\in(0,2D)$ such that $\nu_3/\nu_2>1$, by applying  again \eqref{easycomput},
$$
\gamma_m^{-\delta} m^{-2} \sum_{i\neq j} (\Gamma_{i,j})^2 \lesssim m^{\nu_1\delta} m^{-2}\sum_{i\neq j}  |j-i|^{-\nu_3} \lesssim m^{\nu_2\delta - \nu_3} \vee ( m^{\nu_2\delta-1} \log m)
$$
for any $\delta>0$. We derive \eqref{vanish-secondorder+} because taking $\delta>1$ such that $\delta<\nu_3/\nu_2$ and $\delta<1/\nu_2$ is possible. By using Theorem~\ref{main-thm} under \eqref{assumeasy}, 
 we derive 
\begin{align*}
\gamma_m^{-1/2} ( \G - I)    \leadsto c_1(\cdot) Z  , \mbox{ as $m\rightarrow\infty$,}
  \end{align*}
for $Z\sim\mathcal{N}(0,1)$. This is in accordance with Theorem~1.1 of \cite{DT1989} (see in particular Example~1 therein).

Finally, let us note that Assumption \eqref{vanish-secondorder} of Theorem~\ref{main-thm} is not satisfied for a covariance matrix of the type \eqref{matrix:vanish-long-range} taken with $D\geq 1$ (short-range) (the other terms in the covariance expansion \eqref{equ:covG-detail} are required in the limit, see \cite{CM1996}).  

\paragraph{Weak short/long range correlations}

Let us modify slightly the matrix \eqref{matrix:vanish-long-range}, by letting: 
\begin{equation}\Gamma^{(m)}_{i,j}=\rho_m \:r(|j-i|),\:\:\: \mbox{ for } r(0)=1, \:\:r(k)=k^{-D} ,\:\:\: D> 0,\label{matrix:vanish-range}
\end{equation}
where  $\rho_m$ is some nonnegative parameter (we removed the slowly varying function for the sake of simplicity).  When $\rho_m$ varies in function of $m$, note that the latter is not of the stationary type.
 From \eqref{easycomput}, we have
  \begin{align}
 m\gamma_m \sim \left\{ \begin{array}{cl} 2\rho_m \frac{ m^{1-D}}{(1-D)(2-D)} & \mbox{ if $D\in  [0,1)$}\\ 2\rho_m\log m & \mbox{ if $D=1$}\\2\rho_m \sum_{k\geq 1} k^{-D} & \mbox{ if $D>1$}\\\end{array} \right.\label{thetavanish}
\end{align}
Assuming that the quantity \eqref{thetavanish} as a limit (denoted $\theta$) and that $\rho_m\to 0$ as $m$ grows to infinity, \eqref{vanish-secondorder} and \eqref{vanish-fourthorder} hold if $m^{1+\delta} \rho_m^2 \rightarrow 0$ with $\delta>0$ (because \eqref{assumeasy} holds).
The resulting rate of convergence $r_m$ is given as a function of $D$ and $\rho_m$ in Table~\ref{table:vanish-long-range}. 
Markedly, weak short-range correlations ($D>1$) always yields $r_m\sim m^{1/2}$ while weak long-range correlations ($D<1$) can give both regimes. For instance, taking $\rho_m \sim m^{-2/3}$ yields $r_m\sim m^{D/2+1/3}$ for $D<1/3$ and $r_m\sim m^{1/2}$ otherwise.
Overall, the convergence rate increases with $D$. 

 \begin{table}[h!]
\begin{tabular}{|l|l|l|}
 \hline
 &&\\
  & $D\in[0,1)$  & $D\geq 1$\\
 &&\\\hline
  &&\\
 & $\rho_m  m^{1-D} =O(1)$ & $\theta=0$  \\
$\theta<\infty$ & $r_m \sim \sqrt{m}$   &	 $r_m\sim\sqrt{m}$\\
  &&\\\hline
  &&\\
 & $\rho_m  m^{1-D} \rightarrow \infty$ & \\
$\theta=\infty$ & $r_m\sim \rho_m^{-1/2} m^{D/2}$     &  	 not possible \\
&&\\\hline
  \end{tabular}
\caption{Rate $r_m$ defined by \eqref{rate} in function of $D\geq 0$ and $\rho_m$ such that $\rho_m=o(m^{-(1/2+\delta)})$ for some $\delta>0$, for the particular covariance \eqref{matrix:vanish-range}.
\label{table:vanish-long-range}}
\end{table}

\paragraph{Vanishing factor model}

``Spiked" covariance matrix has been introduced in \cite{John2001}. It assumes that the $k$-first eigenvalues of the covariance matrix are greater than $1$ (for some fixed value of $k$) while the other are all equal to $1$. In our setting where we consider only \textit{correlation} matrices, 
we assume that the sequence of eigenvalues is constant after some fixed rank $k$. Precisely, let us consider a matrix $\Gamma^{(m)}$ of the following form:  
\begin{align}
\Gamma^{(m)} = (1-\rho_{m}) I_m + \rho_{m} P H P^T, \label{matrix:spike}
\end{align}
where $H$ is a $k\times k$ diagonal matrix with diagonal entries $h^{(m)}_1, \dots, h^{(m)}_k \in (1,\infty)$, where $P=(p^{(m)}_{i,r})_{1\leq i \leq m, 1\leq r\leq k}$ is an $m\times k$ matrix such that $P^TP =I_k$ and where $\rho_m\in[-1,1]$ is some parameter. Importantly, $k$ is taken fixed with $m$. The $k$ first eigenvalues of $\Gamma^{(m)} $ are thus given by $1-\rho_m + \rho_m h_r^{(m)}$, $r=1,\dots, k$, while the remaining eigenvalues are all equal to $1-\rho_m$.
Hence, to ensure that $\Gamma^{(m)}$ given by \eqref{matrix:spike} is a well defined correlation matrix, we should additionally assume that for all $r=1,\dots, k$,  $1-\rho_m + \rho_m h_r^{(m)} \geq 0$, 
and that $P H P^T$ has diagonal entries equal to $1$, that is, for all $ i =1,\dots,  m$, $\sum_{r=1}^k h^{(m)}_r (p^{(m)}_{i,r})^2 = 1$. Note that the latter requires $h^{(m)}_1 + \dots + h^{(m)}_k=m$ and thus $\max_r \{ h_r^{(m)} \} \geq m/k$.

Next, by using \eqref{matrix:spike}, the conditions above and some properties of the Frobenius norm, we can derive the following:
\begin{align}
m\gamma_m
&= \rho_m \sum_{r=1}^k h^{(m)}_r \left( m^{-1/2} \sum_{i=1}^m p^{(m)}_{i,r} \right)^2 -\rho_m; \label{equinter1}\\
m^{-2} \sum_{i\neq j} \left(\Gamma^{(m)}_{i,j}\right)^2 &= \rho^2_m  \left(  \sum_{r=1}^k  (h^{(m)}_r/m)^2 -1/m \right) . \label{equinter2}
\end{align}
Since the RHS of \eqref{equinter2} is upper-bounded by $\rho_m^2(k -m^{-1})$ and lower-bounded by $ \rho_m^2(k^{-2} -m^{-1})$ and since $k$ is taken fixed with $m$, condition \eqref{weakdep} is satisfied if and only if $\rho_m\to 0$ while \eqref{vanish-secondorder} holds if and only if $r_m \rho_m \to 0$.
Additionally, we have
\begin{align*}
m^{-2} \sum_{i\neq j} \left(\Gamma^{(m)}_{i,j}\right)^4 &= \rho^4_m  \left( m^{-2} \sum_{r_1,\dots,r_4} h_{r_1} \dots h_{r_4} \left(\sum_{i=1}^m p_{i,r_1} \dots p_{i,r_4}\right)^2 -1/m \right)\\
&\leq \rho^4_m  \left( m^{-2} \left(\sum_{r_1,r_2} h_{r_1} h_{r_2} \sum_{i=1}^m p^2_{i,r_1}p^2_{i,r_2} \right)^2 -1/m \right)\\
&= \rho^4_m (1-1/m),
\end{align*}
where we used the Cauchy-Schwartz inequality (we dropped the dependence in $m$ in the notation for short). Finally, the assumption of 
Theorem~\ref{main-thm} are all fulfilled provided that 
\begin{align}\label{condition-factormodel}
\mbox{$r_m^{2+\delta} \rho^2_m \to 0$ with $\delta>0$}
\end{align} 
(up to consider a subsequence making the quantity into \eqref{equinter1} converges to some $\theta$). In \eqref{condition-factormodel}, the rate $r_m$ can be computed by using the definition, see \eqref{rate}, or expression \eqref{equinter1}. The rate of convergence thus intrinsically  depends on the asymptotic behavior of the coordinate-wise mean of each eigenvector $(p_{i,r}^{(m)})_{1\leq i\leq m}$ . 

To further illustrate this example, we can focus on the particular case where $k=1$.  In that case, the model can be equivalently written  as
\begin{align}
\Gamma^{(m)} = (1-\rho_{m}) I_m + \rho_{m} \xi \xi^T, \label{matrix:spike:kequal1}
\end{align}
where $\xi=\xi^{(m)}$ is a $m\times 1$ vector in $\{-1,1\}^m$ and where $\rho_m \in[ -(m-1)^{-1},1]$.
The model  \eqref{matrix:spike:kequal1} contains as particular instances the equicorrelated matrix ($\xi^{(m)}=(1\: 1 \cdots 1)^T$) and the alternate equicorrelated matrix ($\xi^{(m)}=(1\: -1\: 1 \cdots )^T$) that we have studied above. We easily check that condition \eqref{condition-factormodel} recovers the conditions that we obtained in each of theses particular cases.
In general, for an arbitrary $\xi^{(m)} \in \{-1,1\}^m$, since the quantity in \eqref{equinter1} is equal to
\begin{align}
\rho_m \left( m^{-1/2} \sum_{i=1}^m \xi^{(m)}_{i} \right)^2 -\rho_m,\label{quantityrandom}
\end{align}
the rate $r_m$ is directly related to the number of $-1$ and $+1$ into $\xi^{(m)}$. For instance, if $\xi^{(m)}=(U_1,\dots, U_m)$ where  $U_1,U_2,\dots$ are i.i.d. random signs, we have by the central limit theorem that the quantity \eqref{quantityrandom} tends to $0$ (in probability) whenever $\rho_m \to 0$, which gives a rate $r_m \sim \sqrt{m}$ (in probability).  Hence, we obtain the convergence \eqref{resmainthm} with the same rate and asymptotic variance as in the independent case whenever $m^{1+\delta}\rho_m^2\to 0$ with $\delta>0$.

  \paragraph{Sample correlation matrix}

We consider the model where the correlation matrix is generated \textit{a priori} as a Gaussian empirical correlation matrix. Namely, let us assume that 
\begin{align}
\Gamma^{(m)} = D^{-1} S D^{-1},\mbox{ for } S= n_m^{-1}  X^T X \mbox{ and } D= \mbox{diag}(S_{1,1}, \cdots, S_{m,m})^{1/2} \label{matrix:emp}
\end{align}
where $X$ is a $n_m\times m$ matrix with i.i.d. $\mathcal{N}(0,1)$ entries. Assume $m/n_m \rightarrow 0$ as $m$ tends to infinity, which, in a statistical setup, corresponds to assume that the number $m$ of variables (columns of $X$) is of smaller order than the sample size $n_m$.

A by-product of Theorem~2 in \cite{BY1993} (adding a number of variables which is a vanishing small proportion of the sample size) is that, 
\begin{equation*}
|| S -I_m ||_2 \xrightarrow{P} 0,
\end{equation*}
where $||\cdot||_2$ denotes the Euclidian-operator norm, that is, $||S  -I_m||_2=\max_{1\leq i\leq m} |\lambda^{(m)}_i-1|$ and $\lambda_1^{(m)},\dots,\lambda_m^{(m)}$ denote the eigenvalues of  $S$. Hence, $\max_{1\leq i\leq m} | S_{i,i}-1|\xrightarrow{P} 0$, which in turn implies $|| \Gamma^{(m)}  -I_m ||_2 \xrightarrow{P} 0$.
Next, simple arguments entail the following inequalities: 
\begin{align*}
\left|m^{-1} \sum_{i\neq j} \Gamma^{(m)}_{i,j}\right| &= m^{-1} |< (1\cdots 1)^T , (\Gamma^{(m)}-I_m) (1\cdots 1)^T>| \leq ||\Gamma^{(m)}  -I_m||_2; 
\\
r_m^2 m^{-2} \sum_{i\neq j} \left(\Gamma^{(m)}_{i,j}\right)^2 &\leq m^{-1}\sum_{i=1}^m (\lambda_i^{(m)}-1)^2\leq  ||\Gamma^{(m)}  -I_m||_2^2 ;
\\
r_m^{4+\eps_0} m^{-2} \sum_{i\neq j} \left(\Gamma^{(m)}_{i,j}\right)^4 &\leq  \left\{\min_{1\leq i\leq n_m} |S_{i,i}|\right\}^{-4} m^{\eps_0/2} \sum_{i\neq j} (S_{i,j})^4 .
\end{align*}
Moreover, we easily check that $\E (n_m^{1/2} S_{i,j})^4= \E \left( n_m^{-1/2} \sum_{k=1}^{n_m} X_{k,i} X_{k,j}\right)^4$ is upper bounded by some positive constant. Hence, by assuming that the sequence $n_m$ satisfies
$$
\mbox{$m^{1+\delta}/n_m \rightarrow 0$ for some $\delta>0$},
$$
the above inequalities implies  
that the rate is $r_m\sim \sqrt{m}$, that   \eqref{equtheta} holds  with $\theta=0$ and that \eqref{vanish-secondorder} and \eqref{vanish-fourthorder} are satisfied (all these convergences holding in probability).  Hence, Theorem~\ref{main-thm} can be applied and this shows that the asymptotic of the empirical distribution function is  the same as under independence.

\section{Application to multiple testing}\label{sec:applistat}\label{sec:appli}

\subsection{The curse of dependence}

The so-called ``Benjamini and Hochberg procedure" (BH procedure), widely popularized after the celebrated paper \cite{BH1995}, is often given as the default procedure to provide a false discovery proportion (FDP) close to some pre-specified error level $\alpha$. More specifically, the BH procedure  provides that  the \textit{expectation} of the FDP, called the false discovery rate (FDR), is bounded by $\alpha$ under independence of the tests (and also for some type of positive dependence, see \cite{BY2001}). Furthermore, many authors reported that the FDR of the BH procedure is  essentially unaffected by dependencies, see, e.g., \cite{Far2006,KW2008}. It is therefore tempting to conclude that the BH procedure works whatever the dependencies are. 
However, as noticed by Lehmann and Romano: 
``\textit{control of the FDR does not prohibit the FDP from varying, even if its average value is bounded}", see \cite{LR2005}. In addition, some authors have exhibited that the distribution of the FDP of BH can be wide spread in a particular (unrealistic) equi-correlated framework, by using simulations, see, e.g., Table~2 in \cite{Korn2004} and by using a theoretical study, see \cite{DR2011}.
The present work brings a broad theoretical support for this, by showing that the distribution of the FDP of BH procedure is widening as the quantity $\gamma_m$ defined by \eqref{def:gammam} grows. 

The formal link between the FDP, the BH procedure and e.d.f.'s has been delineated in \cite{GW2004,Far2007} (FDP at a fixed threshold) and consolidated later in \cite{Neu2008} (FDP at BH threshold). 
Here, we follow the approach of \cite{Neu2008}, by using that the FDP of BH procedure is a Hadamard differentiable function of (rescaled) empirical distribution functions. Convergence results are thus derived from Theorem~\ref{main-thm} by applying the (partial) functional delta method, see Proposition~\ref{prop:FDM}.

\subsection{Two-group model, FDP and BH procedure}\label{sec:MTsetting}

 Let us add to the original vector $Y\sim \mathcal{N}(0,\Gamma)$ an unknown vector $H=(H_i)_{1\leq i \leq m}\in\{0,1\}^m$  as follows: for $1\leq i \leq m$,
\begin{align}\label{model}
X_i =  \delta H_i + Y_i ,
\end{align}
for some positive number $\delta$ (assumed to be fixed with $m$). Hence $X\sim \mathcal{N}(\delta H,\Gamma)$.
Now consider the statistical problem of finding $H$  from the observation of $X=(X_i)_{1\leq i \leq m}$. 
From an intuitive point of view, $H$ is the ``signal" (unknown parameter of interest), $Y$ is the ``noise" (unobserved) while $\Gamma$ and $\delta$ are ``nuisance" parameters, generally assumed to be unknown.

Let us define the following e.d.f.'s: for $t\in[0,1]$,
\begin{align}
\Gzero(t)&=m_0^{-1} \sum_{i=1}^m (1-H_i)\ind{\Phi(X_i)\leq t};\label{processF0}\\
\Gone(t)&=m_1^{-1} \sum_{i=1}^m H_i\ind{\Phi(X_i)\leq t}\label{processF1};\\
\Gchap(t)&=m^{-1} \sum_{i=1}^m \ind{\Phi(X_i)\leq t}=\frac{m_0}{m}\Gzero(t)+\frac{m_1}{m}\Gone(t)\label{processG},
\end{align}
where $m_0=\sum_{i=1}^m (1-H_i)$ and $m_1=\sum_{i=1}^m H_i$.  The proportions $m_0/m$ and $m_1/m$ are supposed to converge when $m$ grows to infinity and the limits are denoted by $\pi_0\in(0,1)$ and $\pi_1\in(0,1)$, respectively. From Section~\ref{sec:convcov}, when $\Gamma$ satisfies \eqref{weakdep}, the e.c.d.f.'s $\Gzero(t)$, $\Gone(t)$ and $\Gchap(t)$ converge in probability and we denote in what follows  the limiting c.d.f.'s by $F_0(t)=t$, $F_1(t)=\pnorm(\qnorm(t)-\delta)$ and $G(t)=\pi_0 F_0(t)+\pi_1 F_1(t)$, respectively.

Here, the quality of a procedure that rejects each null hypothesis ``$H_i=0$" whenever $\Phi(X_i)\leq t$
is given by
\begin{equation*}
\FDP_m(t)= \frac{\frac{m_0}{m}\Gzero(t)}{\Gchap(t)}, 
\end{equation*}
 where we used the convention $0/0=0$. Now, define the following functional: for $\alpha\in(0,1)$, 
\begin{align*}
\T(H)&=\sup\{t\in[0,1]\telque H(t)\geq t/\alpha\} \mbox{ for $H\in D(0,1)$},
\end{align*}
with the convention $\sup\{\emptyset\}=0$. Classically, the BH procedure (at level $\alpha$) corresponds the thresholding $\T(\Gchap)$, see \cite{GW2004}.
In the sequel, we study the asymptotic behavior of $\FDP_m(\T(\Gchap))$, denoted by $\FDP_m$ for short.

\subsection{Partial functional delta method}

Since we have $\Gchap(\T(\Gchap))=\T(\Gchap)/\alpha$ a.s., the FDP of BH procedure corresponds to the random variable 
\begin{equation}\label{FDPBH}
\FDP_m= \alpha\frac{\frac{m_0}{m}\Gzero(\T(\Gchap))}{\T(\Gchap)} 
= \Psi\left(\frac{m_0}{m}\Gzero,\frac{m_1}{m}\Gone\right),
\end{equation}
where we used the following functional: 
\begin{align}
\Psi(H_0,H_1)&=\alpha \frac{H_0(\T(H_0+H_1))}{\T(H_0+H_1)},  \mbox{ for $(H_0,H_1)\in D(0,1)^2$},\label{equ-Psi}
\end{align}
still using the conventions $\sup\{\emptyset\}=0$ and $0/0=0$. 
By Corollary~7.12 in \cite{Neu2008}, $\T$ is Hadamard differentiable at function $G$, tangentially to the set  $C(0,1)$ of continuous functions on $(0,1)$ and w.r.t. the supremum norm (we refer to Section~20.2 in \cite{Vaart1998} for a formal definition of Hadamard differentiable functions). This holds because $G$ is strictly concave and $\lim_{t\rightarrow 0} G(t)/t = +\infty$, which yields in particular $\T(G)\in(0,1)$.
As a consequence, standard calculations show that $\Psi$ is Hadamard differentiable at $(\pi_0 F_0,\pi_1 F_1)$ tangentially to $C(0,1)$, with derivative
\begin{align}
\dot{\Psi}_{(\pi_0 F_0,\pi_1F_1)}(H_0,H_1) =&\alpha \frac{H_0(\T(G)) }{\T(G)},\label{expr:der} \mbox{ for $(H_0,H_1)\in C(0,1)^2$.}
\end{align}
Now,  by using \eqref{FDPBH}, the functional delta method provides the asymptotic behavior of $\FDP_m$ from the one of $(\frac{m_0}{m}\Gzero,\frac{m_1}{m}\Gone)$. As a matter of fact, since the derivative $\dot{\Psi}_{(\pi_0 F_0, \pi_1 F_1)}(H_0,H_1)$ only depends on $H_0$ while the limit processes are (a.s.) continuous, establishing convergence results separately for $\Gzero$ and $\Gone $  is sufficient (we do not need to consider  the joint process $(\frac{m_0}{m}\Gzero,\frac{m_1}{m}\Gone)$). 
We have precisely formulated this argument in Proposition~\ref{prop:FDM}. This is an interesting novelty w.r.t. the methodology of \cite{Neu2008}. 
Hence, applying (twice) Theorem~\ref{main-thm} we are able to derive a convergence result for $\FDP_m$. 
\subsection{Results}

First, let us  introduce the following additional quantities:
\begin{align}
r_{0,m}&= \left(m_0^{-1} + \left| m_0^{-2} \sum_{i\neq j} (1-H_i)(1-H_j) \Gamma_{i,j}\right|\right)^{-1/2}\label{rate0};\\
 r_{1,m}&= \left(m_1^{-1} + \left| m_1^{-2} \sum_{i\neq j} H_iH_j \Gamma_{i,j}\right|\right)^{-1/2}\label{rate1}.
\end{align}

\begin{corollary}\label{cor-FDP}
Consider the two-group model \eqref{model}, generated from parameters $\delta$, $H=H^{(m)}$ and a correlation matrix $\Gamma=\Gamma^{(m)}$. 
Assume that $m_0$ (depending on $H$) is such that $\sqrt{m} (m_0/m-\pi_0)\rightarrow 0$.
Assume that $\Gamma$ satisfies either $\{$\eqref{vanish-secondorder} and \eqref{vanish-fourthorder}$\}$ or $\{$\eqref{vanish-secondorder+} and \eqref{gammam+}$\}$.
Assume that the rates $r_m$, $r_{0,m}$ and $r_{1,m}$, respectively defined by \eqref{rate}, \eqref{rate0}  and \eqref{rate1}, grow proportionally to infinity as $m$ tends to infinity. Let $\alpha\in(0,1)$ and $t^\star=t^\star(\delta,\alpha)$ be the unique $t\in(0,1)$ such that  $G(t)=t/\alpha$. Let $h(t^\star)= (\phi(\Phi^{-1}(t^\star))/t^\star)^2$.
Then the sequence of r.v. $\FDP_m$ defined by \eqref{FDPBH} enjoys the following convergence:
\begin{align}
\frac{\FDP_m - \pi_0\alpha}{\pi_0\alpha\left\{(1/t^\star-1)/m_0 +   h(t^\star)   \gamma_{0,m} \right\}^{1/2}}\leadsto  \mathcal{N}(0,1)\label{res:cor-FDP},
  \end{align}
where $\gamma_{0,m}= m_0^{-2}\sum_{i\neq j}  (1-H_i)(1-H_j)\Gamma_{i,j}$.
\end{corollary}

\begin{proof}
First, classically, it is sufficient to prove that the convergence \eqref{res:cor-FDP} holds up to consider a subsequence.
Hence, we can assume that  \eqref{equtheta} and the convergences 
\begin{align}
  m_0^{-1} \sum_{i\neq j} (1-H_i)(1-H_j) \Gamma_{i,j}&\rightarrow \theta_0 ;\label{equtheta0}\\
 m_1^{-1} \sum_{i\neq j} H_i H_j \Gamma_{i,j}&\rightarrow \theta_1;\nonumber
\end{align}
 hold, with $\theta$, $\theta_0$ and $\theta_1$ valued in $[-1,+\infty]$.
Also note that since $r_m \propto r_{0,m}$ (resp. $r_m \propto r_{1,m}$), the sub-matrices $(\Gamma_{i,j})_{i,j: H_i=H_j=0}$ and $(\Gamma_{i,j})_{i,j: H_i=H_j=1}$ satisfies the same assumption set as $\Gamma$.
Now, let us write
\begin{align}\label{jolieequation}
r_{0,m}\left(\frac{m_0}{m}\Gzero(t) - \pi_0 F_0(t)\right) =  r_{0,m} m^{-1} \sum_{i=1}^m (1-H_i) (\ind{\Phi(X_i)\leq t} - t) + r_{0,m}t (m_0/m- \pi_0) .
  \end{align}
In the RHS of \eqref{jolieequation}, while the second term converges to $0$ by assumption, a consequence of Theorem~\ref{main-thm} is that the first term converges to a process with covariance function
$$\pi_0^2 \left[ \frac{1}{1+|\theta_0|} (t\wedge s -t s)+ \frac{\theta_0}{1+|\theta_0|} c_1(t) c_1(s) \right], \:\:\: \mbox{for all $t,s\in[0,1]$}.$$
Obviously, a similar result holds for the process $r_{1,m}( \frac{m_1}{m}\Gone - \pi_1 F_1)$. 

Applying the (partial) functional delta method as explained in Proposition~\ref{prop:FDM} (by using $ r_{0,m} \propto r_{1,m}$ and  \eqref{expr:der}), we obtain
\begin{align}
r_{0,m} ( \FDP_m - \pi_0\alpha )    \leadsto  \mathcal{N}\left(0\:,(\alpha \pi_0)^2 \left[ \frac{1/t^\star-1}{1+|\theta_0|} + \frac{\theta_0}{1+|\theta_0|} (c_1(t^\star)/t^\star)^2\right] \right)\label{res:cor-FDP2}.
  \end{align}
 Finally, we easily derive \eqref{res:cor-FDP} by separating the cases $\theta_0<+\infty$ and $\theta_0=+\infty$.
\end{proof}

As an illustration, Corollary~\ref{cor-FDP} can be used  in the independent case ($\gamma_{0,m}=0$) or $\rho_m$-equi-correlated case ($\gamma_{0,m}=\rho_m$), so recovering the previous results of \cite{Neu2008,Neu2009} (in the Gaussian case)  and \cite{DR2011}, respectively. 
This holds for any $H$ satisfying $\sqrt{m} (m_0/m-\pi_0)\rightarrow 0$.
Note that, in general, the quantity $\gamma_{0,m}$  depends on the unknown $H$ and not only on $\Gamma$. Hence, the asymptotic properties of $\FDP_m$ potentially depends on which null hypotheses are true or not, which can be considered as a limitation. Nevertheless, this fact is inherent to the multiple testing setting considered here, because the dependencies accounting in the FDP of BH's procedure are related to the sub-matrix $(\Gamma_{i,j})_{i,j: H_i=H_j=0}$ and thus are linked to the location of the true null hypotheses. \\

A convenient way to circumvent this problem is to add \textit{prior random effects}, by assuming that, previously and independently to the model \eqref{model}, we have drawn $H=(H_1,\dots,H_m)$ for $H_1,H_2,\dots$  i.i.d. Bernoulli variables of parameter $\pi_1=1-\pi_0$, for some $\pi_0\in(0,1)$. 
Thus $X$ follows the distribution $\mathcal{N}(\delta H,\Gamma)$ \textit{conditionally} on $H$.
The corresponding global (unconditional) model, often referred to as the \textit{two-group mixture model} has been widely used in the multiple testing literature, see, e.g. \cite{ETST2001,Storey2003,GW2004,RV2010}. By contrast with the previous model, $H$ is random. In particular, $m_0 =\sum_{i=1}^m (1-H_i)\sim \mathcal{B}(m,\pi_0)$ and $\sqrt{m} (m_0/m-\pi_0)$ does not degenerate at the limit, which adds some extra variance in the FDP convergence result. 
The counterpart is that the statement is substantially simplified, as we can see below.

\begin{corollary}\label{cor-FDP2}
Consider the two-group mixture model defined above, generated from parameters $\delta>0$, $\pi_0\in(0,1)$ and a correlation matrix $\Gamma=\Gamma^{(m)}$. 
Assume that $\Gamma$ satisfies either $\{$\eqref{vanish-secondorder} and \eqref{vanish-fourthorder}$\}$ or $\{$\eqref{vanish-secondorder+} and \eqref{gammam+}$\}$.
Let $\alpha\in(0,1)$ and $t^\star=t^\star(\delta,\alpha)$ be the unique $t\in(0,1)$ such that  $G(t)=t/\alpha$. Let $h(t^\star)= (\phi(\Phi^{-1}(t^\star))/t^\star)^2$.
Then the sequence of r.v. $\FDP_m$ defined by \eqref{FDPBH} enjoys the following convergence:
\begin{align}
\frac{\FDP_m - \pi_0\alpha}{\pi_0\alpha\left\{ (1/ t^\star-\pi_0)/(\pi_0 m) +   h(t^\star)   \gamma_{m} \right\}^{1/2}}\leadsto  \mathcal{N}(0,1)\label{res:cor-FDP2},
  \end{align}
where $\gamma_{m}$ is defined by \eqref{def:gammam}. 
\end{corollary}

\begin{proof}
Again, it is sufficient to state the result up to consider a subsequence. Thus \eqref{equtheta} holds without loss of generality.
First check that \eqref{vanish-secondorder} entails 
  \begin{align}\label{equ-truc}
\frac{r_m^2}{m^2} \left( \sum_{i\neq j} (1-H_i)(1-H_j) \Gamma_{i,j} - \pi_0^2  \sum_{i\neq j} \Gamma_{i,j}\right)= o_P(1),
  \end{align}
 (computing, e.g., the variance of the latter) and this convergence can be made a.s. 
 by taking a suitable subsequence. 
 A consequence of \eqref{equ-truc} is that $\gamma_{0,m} \sim \gamma_m$ a.s.
 (in particular, $\theta_0$ defined by \eqref{equtheta0} equals $\pi_0\theta$.)
 This implies $r_m \propto r_{0,m}$  (a.s.) and thus the adequate assumption set for the sub-matrices $(\Gamma_{i,j})_{i,j: H_i=H_j=0}$ and $(\Gamma_{i,j})_{i,j: H_i=H_j=1}$. Now, by using \eqref{jolieequation}, we obtain that $r_{0,m}\left(\frac{m_0}{m}\Gzero(t) - \pi_0 F_0(t)\right)$ converges (unconditionally) to a process with covariance function defined by: for all $t,s\in[0,1]$,
\begin{align*}
&\pi_0^2 \left[ \frac{1}{1+\pi_0 |\theta|} (t\wedge s -t s)+ \frac{\pi_0 \theta}{1+\pi_0 |\theta|} c_1(t) c_1(s) \right] + \frac{\pi_0(1-\pi_0)}{1/\pi_0+|\theta|} ts\\
& = \pi_0^2 \left[ \frac{1}{1+\pi_0 |\theta|} (t\wedge s -\pi_0 t s)+ \frac{\pi_0 \theta}{1+\pi_0 |\theta|} c_1(t) c_1(s) \right]
 \end{align*}
  Obviously, a similar result holds for the process $r_{1,m}( \frac{m_1}{m}\Gone - \pi_1 F_1)$. 
We finish the proof by applying the (partial) functional delta method, see Proposition~\ref{prop:FDM}. \end{proof}
 
 \subsection{Discussion}
 
Corollary~\ref{cor-FDP2} provides a theoretical support for the ``curse of dependence" of BH procedure: as $m$ grows to infinity, the concentration of $\FDP_m$ around $\pi_0\alpha$ deteriorates when $\gamma_m$ increases, so when positive correlations appear between the individual statistical tests.
However, notice that, perhaps surprisingly, negative correlations help to decrease $\gamma_m$ and can yields to a concentration even better than under independence  when $\gamma_m$ is negative (although this phenomenon is necessary of limited amplitude because  $\gamma_m\geq -1/m$).
 
To illustrate further Corollary~\ref{cor-FDP2}, Figure~\ref{fig:conclusion}  displays the true distribution of $\FDP_m$, together with the Gaussian approximation obtained by Corollary~\ref{cor-FDP2}.
The two-group mixture model
chosen to generate the $X_i$'s uses a factor model \eqref{matrix:spike} for $\Gamma$ with the following parameters: $k=3$, $m\rho_m\in\{0,10,10^2,10^3\}$, $h_1/m=0.4$, $h_2/m=0.3$, $h_3/m=0.6$, and 
$p_{1}=(1,1,\dots,1)/m^{1/2}$, $p_{2}=(1,-1,1,-1,\dots,1,-1)/m^{1/2}$, $p_{3}=(1,1,\dots,1,-1,-1,\dots,-1)/m^{1/2}$. The parameters of the mixture are $\pi_0=0.9$ and $\delta=3$. The BH procedure is taken at level $\alpha=0.25$.

This experiment shows that, even for a relatively small values for $\rho_m$ ($\rho_m=0.002$ or $\rho_m=0.02$), the FDP distribution can be largely affected by the dependencies. Also, for $m=5\,000$ (left picture), while the Gaussian approximation looks accurate for $m\rho_m\in \{0,10,100\}$, this seems more questionable when $m\rho_m=1\,000$. This non-Gaussian phenomenon, whose amplitude increases with $\rho_m$ (for a fixed $m$), shows the limit of the proposed methodology. As a matter of fact, additional experiments show that the approximation induced by Theorem~\ref{main-thm} is still valid for $m=5\,000$ and $m\rho_m=1\,000$. As a consequence, we believe that the observed bias comes from the functional delta method, because the functional $\Psi$ \eqref{equ-Psi} cannot be considered as linear in that case. Finally, the right display in Figure~\ref{fig:conclusion} shows that, as one can expect, this phenomenon disappears by increasing the value of $m$.
  
  \begin{figure}
\begin{center}
\begin{tabular}{cc}
$m=5\,000$ & $m=50\,000$\\
\includegraphics[scale=0.4]{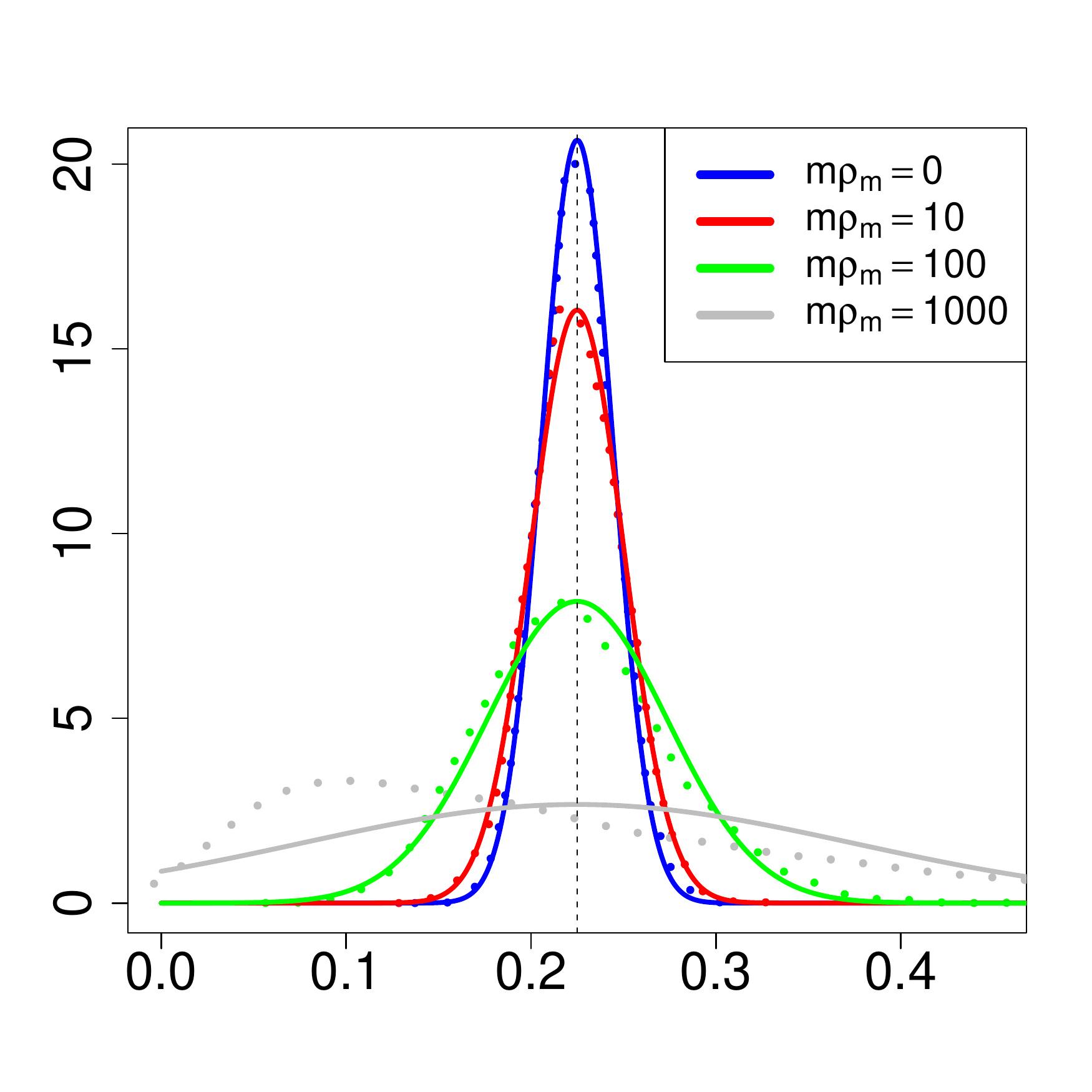} & \includegraphics[scale=0.4]{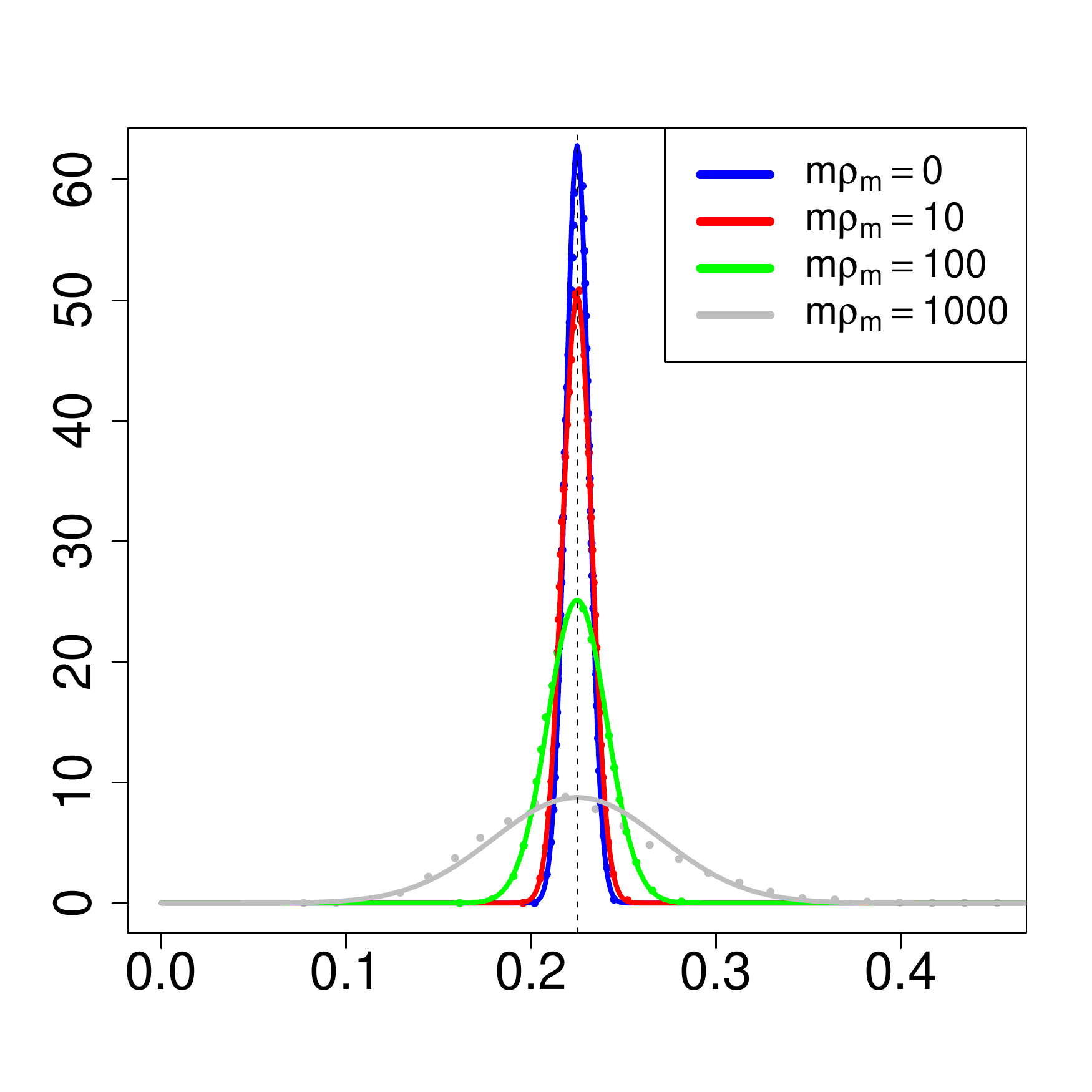}
\end{tabular}
\vspace{-1cm}
\caption{The curse of dependence for BH procedure: distribution of $\FDP_m$ \eqref{FDPBH}; the dotted lines corresponds to the true distribution computed over $5\,000$ simulations, the solid lines display the Gaussian approximation given by Corollary~\ref{cor-FDP2} (whose mean, $\pi_0\alpha=0.225$, is displayed by the dashed vertical line). Simulations made in a $3$-factor model, see text.}
\label{fig:conclusion}
\end{center}
\end{figure}

This study reinforces the idea that the BH procedure should be used very carefully when there are dependencies between the individual tests. Following the work of \cite{RW2007}, an interesting task would be to correct the BH procedure by  taking into account these dependencies while still providing a valid control of the FDP.  This is an exciting direction for a future work.

\section{Proof of Theorem~\ref{main-thm}}\label{sec:proof}

\subsection{A related result and additional notation}

Let us define the ``modified" empirical distribution  function $\Gtilde$ by the following relation: for  $t\in[0,1]$,
\begin{align}\label{rel-G-Gtilde}
r_m(\G(t)-t)=r_m(\Gtilde(t)-t)  + c_1(t) r_m\ol{Y}_m.
\end{align}
The convergence of the two processes $r_m (\G-I)$ and $r_m(\Gtilde-I)$ are strongly related by \eqref{rel-G-Gtilde}. The main idea of our proof is to deduce the convergence of $r_m (\G-I)$ from the one of $r_m(\Gtilde-I)$.
Precisely, the following result will be proved together with Theorem~\ref{main-thm} in the sequel.

\begin{proposition}\label{main-thm2}
Under one of the two sets of assumptions of Theorem~\ref{main-thm}, let us consider the corrected empirical distribution function $\Gtilde$ defined by \eqref{rel-G-Gtilde} and a continuous process $(\wt{\Z}_t)_{t\in[0,1]}$ with covariance function $\wt{K}$ defined by \eqref{equKmin}.
Then we have the convergence (in the Skorokhod topology) 
\begin{align}
r_m ( \Gtilde - I)    \leadsto \wt{\Z}/(1+|\theta|)^{1/2}  , \mbox{ as $m\rightarrow\infty$,} \label{conv-Gtilde}
  \end{align}
where $I(t)=t$ denotes the identity function. 
\end{proposition}

Additionally, throughout the section, we use the following notation
\begin{align}
h_t(x)=\ind{\Phi(x)\leq t} - t - c_1(t) x,\label{defht}
\end{align}
so that $\Gtilde(t)-t =m^{-1} \sum_{i=1}^m h_t(Y_i)$.
Finally, we will sometimes use the following assumption:
\begin{align}\tag{eigenvalues-away$0$}\label{eigenvalues}
\mbox{there exists $\eta>0$ (independent on $m$) lower bounding the $m$ eigenvalues of $\Gamma^{(m)}$}.
\end{align}

\subsection{Convergence of finite dimensional laws for $\Gtilde$}\label{sec:prooffinitedim}

Let us prove the following result.
\begin{proposition}\label{propGtilde} 
Assume that the covariance matrix $\Gamma$ depends on $m$ in such a way that \eqref{vanish-secondorder} holds with $r_m$ defined by \eqref{rate} and assume \eqref{equtheta}. Consider a continuous process $(\wt{\Z}_t)_{t\in[0,1]}$ with covariance function $\wt{K}$ defined by \eqref{equKmin}.
Then, 
the process $(r_m(\Gtilde-I),  Y_1^{(m)})$ (jointly) converges to $\mtc{L}(\wt{\Z}/(1+|\theta|)^{1/2}  )\otimes \mtc{N}(0,1)$ in the sense of the finite dimensional convergence. 
In particular, the convergence \eqref{conv-Gtilde} holds in the sense of the finite dimensional convergence.
\end{proposition}

\begin{proof}

The proof is based on central limit theorems for martingale arrays as presented, e.g., in Chapter~3 of \cite{HH1980}. 

First, since we aim at obtaining a convergence jointly with $Y_1^{(m)}$, a (somewhat technical) but useful task is to define the array of random variables $(Y^{(m)}_i, 1\leq i \leq m, m\geq 1)$ is such a way that $Y_1^{(m)}$ is fixed with $m$. 
This is possible by first considering some variable $Z\sim\mtc{N}(0,1)$, by letting $Y^{(m)}_1=Z$ for all $m\geq 1$, and then by choosing for each $m\geq 2$, the variables $Y^{(m)}_i, 2\leq i \leq m,$ such that 
\begin{itemize}
\item[-] $(Z,Y^{(m)}_i, 2\leq i \leq m)\sim \mtc{N}(0,\Gamma^{(m)})$;
\item[-] $\{(Y^{(m)}_i)_{2\leq i \leq m}, m\geq 2\}$ is a family of mutually independent vectors conditionally on $Z$.
\end{itemize}
This also define a common underlying space $(\Omega,\mtc{F},\P)$ for the array of random variables. 

Now, define the following nested array of $\sigma$-field: for $m\geq 1$,  $\mtc{G}_{m,0}=\sigma(\emptyset)$ and for $1\leq i \leq m$,  
$$
\mtc{G}_{m,i} = \sigma ( Y^{(\l)}_j  , 1\leq j\leq i\wedge \l, 1\leq \l\leq m ).
$$
Next, let us consider for each $t\in[0,1]$, the martingale array $(M_{m,i}(t), \mtc{G}_{m,i}, 1\leq i \leq m, m\geq 1)$ defined as follows:
\begin{align}\label{def:mart}
M_{m,i}(t) &=  \sum_{j=1}^i    X_{m,j}(t)\:\:\mbox{ for } X_{m,j}(t)=\frac{r_m}{m} \left(h_t(Y^{(m)}_j) -  \E \left( h_t(Y^{(m)}_j) \cond \mathcal{G}_{m,j-1}\right)\right).
\end{align}
 Clearly, 
\begin{align}\label{decomp:mart}
 r_m(\Gtilde(t)-t) = M_{m,m}(t) + \frac{r_m}{m} \sum_{i=1}^m \E \left( h_t(Y_i^{(m)}) \cond \mathcal{G}_{m,i-1}\right).
\end{align}
Also note that we can replace each $\mathcal{G}_{m,i}$ by $\mathcal{F}_{m,i}= \sigma (Y^{(m)}_1,\dots,Y^{(m)}_i)$ ($\mathcal{F}_{m,0}=\sigma(\emptyset)$) in the above expression, because $(Y^{(m)}_i, 2\leq i \leq m)$ is independent of $(Y^{(\l)}_j, 2\leq j\leq i \wedge \l, 2\leq \l<m)$, conditionally on $Y^{(m)}_1$. 

\paragraph{Case 1: \eqref{eigenvalues} is assumed}

We show in Lemma~\ref{lemmacompens} expression \eqref{conv2} that the second term in the RHS of \eqref{decomp:mart} has a vanishing variance as $m$ tends to infinity. Therefore, it remains to show that the conclusion of Proposition~\ref{propGtilde} holds for the process $M_{m,m}$, which we prove 
by using Lindeberg's theorem. We use Corollary~3.1 page 58 in \cite{HH1980} (or more precisely its generalization to the multidimensional case). The conditions are as follows:
\begin{itemize}
\item[(i)] for all $t\in[0,1]$, for all $\eps>0$,  $\sum_{i=1}^m \E \left( (X_{m,i}(t))^2 \ind{|X_{m,i}(t)|>\eps}  \cond \mtc{F}_{m,i-1} \right) \xrightarrow{P} 0$;
\item[(ii)] for all $t,s\in[0,1]$,    $\sum_{i=1}^m \E (X_{m,i}(t)X_{m,i}(s) \cond \mtc{F}_{m,i-1} )\xrightarrow{P} \wt{K}(t,s).$
\end{itemize}
To check (i), let us fix $t\in[0,1]$ and prove $\sum_{i=1}^m \E ( X_{m,i}(t))^4  = o(1)$. By definition, we have
\begin{align*}
\sum_{i=1}^m \E  (X_{m,i}(t))^4 &= \frac{r_m^4}{m^4} \sum_{i=1}^m \E \left( h_t(Y_i^{(m)}) -  \E \big( h_t(Y_i^{(m)}) \cond \mathcal{F}_{m,i-1}\big) \right)^4\\
&\leq 2^4 \left( \frac{r_m^4}{m^3} m^{-1} \sum_{i=1}^m \E  \left(h_t(Y_i^{(m)})\right)^4  + \frac{r_m^4}{m^4}  \sum_{i=1}^m \E \left( \E \big( h_t(Y_i^{(m)}) \cond \mathcal{F}_{m,i-1}\big) \right)^4 \right)\\
&\leq 2^5 \frac{r_m^4}{m^3} m^{-1} \sum_{i=1}^m \E  \left(h_t(Y_i^{(m)})\right)^4.
\end{align*}
Now, the RHS of the previous display converges to zero because $r_m\leq \sqrt{m}$ and $\E  (h_t(Y_i^{(m)}))^4<\infty$. 
 This proves condition (i) of Lindeberg's theorem.

Let us now turn to condition (ii). For $t,s\in[0,1]$, we obviously obtain
\begin{align}
\sum_{i=1}^m \E (X_{m,i}(t)X_{m,i}(s) \cond \mtc{F}_{m,i-1}) =& \frac{r_m^2}{m^2} \sum_{i=1}^m \E (h_t(Y_i^{(m)})h_s(Y_i^{(m)})\cond \mtc{F}_{m,i-1})\nonumber \\
&-  \frac{r_m^2}{m^2} \sum_{i=1}^m \E (h_t(Y_i^{(m)})\cond \mtc{F}_{m,i-1})\E (h_s(Y_i^{(m)})\cond \mtc{F}_{m,i-1}).\label{equRHS}
\end{align}
Next, by using $ab\leq 2(a^2+b^2)$ for all $a,b\in\R$ together with
 \eqref{conv1}, the second term in the RHS of  \eqref{equRHS} tends to zero in probability. Moreover, we have 
\begin{align*}
&\var\bigg(\frac{r_m^2}{m^2} \sum_{i=1}^m \big((h_t(Y_i^{(m)})h_s(Y_i^{(m)})-\E (h_t(Y_i^{(m)})h_s(Y_i^{(m)})\cond \mtc{F}_{m,i-1})\big)\bigg)\\
&=\frac{r_m^4}{m^4}  \sum_{i=1}^m\var\big((h_t(Y_i^{(m)})h_s(Y_i^{(m)})-\E (h_t(Y_i^{(m)})h_s(Y_i^{(m)})\cond \mtc{F}_{m,i-1})\big),
\end{align*}
because the elements inside the sum are martingale increments. Hence, the quantity inside the above display tends to zero. Combining the latter with  \eqref{equRHS} establishes condition (ii) of Lindeberg's theorem  provided that the following holds:
\begin{align*}
\frac{r^2_m}{m^2} \sum_{i=1}^m h_t(Y_i^{(m)})h_s(Y_i^{(m)}) \xrightarrow{P} (1+|\theta|)^{-1}\wt{K}(t,s).
\end{align*}
This comes directly from the law of large number stated in Lemma~\ref{LLN}, because $r_m^2/m\rightarrow (1+|\theta|)^{-1}$ by \eqref{rate} and \eqref{equtheta}.

Applying Lindeberg's theorem (in the underlying space described above), for any $t_1,\dots,t_k\in[0,1]$, the random vector 
$$Z_m=(M_{m,m}(t_1), \dots, M_{m,m}(t_k))$$ converges stably  in the following sense (see, e.g., \cite{JS2003} Definition~5.28): for all (fixed) bounded random variable $U$ and continuous bounded function $f$ in $\R^k$, 
$$
\E (U f(Z_m)) \to \E(U) \E(f(Z)) \mbox{ as $m\to\infty$,}
$$
where  $Z$ is a centered multivariate Gaussian vector with covariance $(1+|\theta|)^{-1} (\wt{K}(t_i,t_j))_{1\leq i,j\leq k}$. This implies that 
$(Z_m,Y_1)$ converges (jointly) in distribution to $ \mtc{L}(Z)\otimes\mathcal{N}(0,1)$. 
This finishes the proof of Proposition~\ref{propGtilde} in the case where  \eqref{eigenvalues} is assumed to hold.

\paragraph{Case 2: \eqref{eigenvalues} is not assumed}

The strategy is to apply Lemma~\ref{lemapproxeps} in order to reduce the study to ``Case 1" above.
For any $\eps>0$, let
$$
{Y}^\eps_i = \frac{Y_i + \eps \xi_i}{(1+\eps^2)^{1/2}},
$$
where $\xi_1,\xi_2,\dots$ are i.i.d. $\mtc{N}(0,1)$ variables, independent of all the $Y_i$'s. The covariance matrix of $({Y}^\eps_1 ,\dots,{Y}^\eps_m )$ is obviously
$${\Gamma}^\eps = \frac{\eps^2}{1+\eps^2} I_m +  \frac{1}{1+\eps^2}\Gamma.$$
Clearly, the corresponding rate \eqref{rate} is 
$ r_m^\eps= \left(m^{-1} + (1+\eps^2)^{-1} \left| \gamma_m\right| \right)^{-1/2}$. It is related to $r_m$ via the following inequalities: $ r_m\leq r_m^\eps \leq (1+\eps^2)^{1/2} r_m$.
Hence, $\Gamma^\eps$ satisfies \eqref{vanish-secondorder} and  \eqref{equtheta} with $\theta$ replaced by $\theta^\eps=\frac{1}{1+\eps^2}\theta$.
 Since it also satisfies \eqref{eigenvalues}, by using Proposition~\ref{propGtilde} in the ``Case 1" above, it satisfies for any $t_1,\dots,t_k\in[0,1]$,
 \begin{itemize}
 \item[(a)] $\big(r_m^\eps (\Gtilde^\eps(t_1)-t_1),\dots, r_m^\eps (\Gtilde^\eps(t_k)-t_k) ,  Y_1^\eps\big)  \leadsto \mtc{L}\left(\frac{(\wt{\Z}(t_1),\dots,\wt{\Z}(t_k))}{(1+|\theta^\eps|)^{1/2}} \right)\otimes \mtc{N}(0,1)$,
\end{itemize}
where $\Gtilde^\eps(t)-t=m^{-1} \sum_{i=1}^m h_t(Y^\eps_i)$ for all $t$.
Next, we clearly have, 
\begin{itemize}
 \item[(b)] $\frac{(\wt{\Z}(t_1),\dots,\wt{\Z}(t_k))}{(1+|\theta^\eps|)^{1/2}} \leadsto  \frac{(\wt{\Z}(t_1),\dots,\wt{\Z}(t_k))}{(1+|\theta|)^{1/2}}$ \mbox{ as $\eps \to 0$}.
\end{itemize}
Let us now prove that for any $t\in[0,1]$, 
\begin{align}\label{equfinalpreuve}
\limsup_m \left\{ \E \left| r_m (\Gtilde(t)-t)-r_m^\eps (\Gtilde^\eps(t)-t) \right| \right\} \to 0 \mbox{ as $\eps \to 0$}.
\end{align}
This will conclude the proof by applying Lemma~\ref{lemapproxeps}. First, we write
\begin{align*}
 &\E \left| r_m (\Gtilde(t)-t)-r_m^\eps (\Gtilde^\eps(t)-t) \right| \\
 \leq&\:  \E \left|r_m/m \sum_{i=1}^m   (h_t(Y_i) - h_t(Y_i^\eps)) \right|+  (r_m^\eps-r_m) \E \left|m^{-1} \sum_{i=1}^m    h_t(Y_i^\eps) \right|\\
 \leq&\:  \left\{ (r_m/m)^2 \:\E \left( \sum_{i=1}^m  (h_t(Y_i) - h_t(Y_i^\eps))\right)^2 \right\}^{1/2} +  ((1+\eps^2)^{1/2}-1) \E \left| r_m^\eps (\Gtilde^\eps(t)-t)  \right|.
\end{align*}
By taking the $\limsup$ in the above display, it only remains to show 
\begin{equation}\label{onestepmore}
\limsup_m \left\{(r_m/m)^2 \:\E \left( \sum_{i=1}^m  (h_t(Y_i) - h_t(Y_i^\eps))\right)^2 \right\} \to 0 \mbox{ as $\eps \to 0$}.
\end{equation}
This can be proved  by using Lemma~\ref{conseqMelher} \eqref{equ-2fact} as follows:
\begin{align*}
 &(r_m/m)^2 \:\E \left( \sum_{i=1}^m  (h_t(Y_i) - h_t(Y_i^\eps))\right)^2 \\
 &= (r_m/m)^2 \sum_{i,j=1}^m  \E \left((h_t(Y_i) - h_t(Y_i^\eps))(h_t(Y_j) - h_t(Y_j^\eps))\right)\\
 &=(r_m/m)^2 \sum_{i,j=1}^m  \left( \E \left(h_t(Y_i)h_t(Y_j)\right)  - \E \left(h_t(Y_i)h_t(Y_j^\eps)\right) - \E \left(h_t(Y_i^\eps)h_t(Y_j)\right) +\E \left(h_t(Y_i^\eps)h_t(Y_j^\eps)\right) \right) \\ 
 &=(r_m/m)^2 \sum_{i,j=1}^m \sum_{\l\geq 2} \frac{(c_\l(t))^2}{\l!} (\Gamma_{i,j})^{\l} \left( 1+ (1+\eps^2)^{-\l} - 2 (1+\eps^2)^{-\l/2}\right)
  \end{align*}
because $\cov(Y_i,Y_j)=\Gamma_{i,j}$, $\cov(Y^\eps_i,Y_j)=\cov(Y_i,Y^\eps_j)=\Gamma_{i,j}/(1+\eps^2)^{1/2}$ and $\cov(Y^\eps_i,Y^\eps_j)=\Gamma_{i,j}/(1+\eps^2)$.
Next, by separating the case $i=j$ and $i\neq j$, the previous display can be upper bounded by
\begin{align*}
& \sum_{\l\geq 2} \frac{(c_\l(t))^2}{\l!}  \left| 1+ (1+\eps^2)^{-\l} - 2 (1+\eps^2)^{-\l/2}\right| + (r_m/m)^2 \sum_{i\neq j} (\Gamma_{i,j})^{2} \times 4  \sum_{\l\geq 2} \frac{(c_\l(t))^2}{\l!}.
  \end{align*}
While the first term above does not depend on $m$ and converges to zero as $\eps\to 0$, the second term above as a $\limsup_m$ equal to zero by \eqref{vanish-secondorder}. This implies \eqref{onestepmore} and finishes the proof.
\end{proof}

\subsection{Convergence of finite dimensional laws for $\G$}\label{sec:prooffinitedimforGchap}

In this section, we aim at proving the following result:
\begin{proposition}\label{propGchap} 
Consider the assumptions of Proposition~\ref{propGtilde}. Then, \eqref{resmainthm} holds  in the sense of the finite dimensional convergence. \end{proposition}

\begin{proof}

From expression \eqref{rel-G-Gtilde}, we investigate the (joint) convergence of $(r_m(\Gtilde-I),  r_m\ol{Y}_m)$. 

\paragraph{Case 1: $\theta=-1$}

In that case, $r_m^2 \var(\ol{Y}_m) \rightarrow 0$. Hence,  we can  directly use Proposition~\ref{propGtilde} to state that $(r_m(\Gtilde-I),  r_m\ol{Y}_m)$ converges to $\mtc{L}(\wt{\Z}/(1+|\theta|)^{1/2}  )\otimes \delta_0$ in the sense of the finite dimensional convergence. This establishes Proposition~\ref{propGchap} in that case.

\paragraph{Case 2: $\theta>-1$} 

Now,  $r_m^2 \var(\ol{Y}_m) $ is converging to some positive real number, namely $(1+\theta)/(1+|\theta|)>0$.
In particular, $\var(\ol{Y}_m) >0$ for $m$ large enough.
Let us define the random variable  $$Y_0= \ol{Y}_m (\var \ol{Y}_m)^{-1/2}.$$ We now consider the $(m+1)$-dimensional random vector $(Y_i)_{0\leq i \leq m}$, which is centered, with a covariance matrix denoted $\Lambda^{(m+1)}=(\Lambda^{(m+1)}_{i,j})_{0\leq i,j\leq m}$ and such that $\Lambda^{(m+1)}_{0,0}=1$, $\Lambda^{(m+1)}_{i,j}=\Gamma^{(m)}_{i,j}$ for $1\leq i,j\leq m$. We easily check that $\Lambda^{(m+1)}$ satisfies \eqref{vanish-secondorder} and \eqref{equtheta} with the same value of $\theta$ and a rate asymptotically equivalent to the original $r_m$, see Lemma~\ref{lemmalambda}. 
Hence, Proposition~\ref{propGtilde} shows that (by using notation therein),
$$
\left( r_m\left( (m+1)^{-1} \sum_{i=0}^m h_t(Y_i)\right),  Y_0\right) \leadsto \mtc{L}(\wt{\Z}/(1+|\theta|)^{1/2}  )\otimes \mtc{N}(0,1),
$$
in the sense of the finite dimensional convergence. Since $r_m h_t(Y_0)/m$ tends to zero in probability, the last display can be rewritten as 
$$
\left( r_m(\Gtilde-I),  \ol{Y}_m (\var \ol{Y}_m)^{-1/2} \right) \leadsto \mtc{L}(\wt{\Z}/(1+|\theta|)^{1/2}  )\otimes \mtc{N}(0,1).
$$
Finally, since $r_m^2 \var(\ol{Y}_m) \rightarrow (1+\theta)/(1+|\theta|)$, we finish the proof by applying \eqref{rel-G-Gtilde}.
\end{proof}

\subsection{Tightness under \eqref{vanish-secondorder}, \eqref{vanish-fourthorder} and \eqref{equtheta}}

To complete the proof of Proposition~\ref{main-thm2}, we prove that the process $X_m=r_m(\Gtilde-I)$ is tight in the Skorokhod space. 
This also implies tightness for $r_m(\G-I)$ by \eqref{rel-G-Gtilde} because $c_1$ is a continuous function on $[0,1]$, itself entailing Theorem~\ref{main-thm}.

We consider here the set of assumptions \eqref{vanish-secondorder}, \eqref{vanish-fourthorder} and \eqref{equtheta} (the second set of assumptions is examined in Section~\ref{sec:autrecond}). 
For proving the tightness of $X_m$, we use Proposition~\ref{prop:tight}. This is possible because $|c_1(t)-c_1(s)|\leq L |t-s|^{1/2}$, $0\leq s, t \leq 1$ for some constant $L>1$ (see Lemma \ref{lem:majc1}). 
Below, we prove that \eqref{tight-crit} holds in the following way: for large $m$,
\begin{equation}\label{tight-crit-applied}
\E \big| X_m(t)-X_m(s) \big|^4 \leq   C\big( |t-s|^{3/2} + (r_m)^{-\eps_0} |t-s| \big), \mbox{ for all $t,s\in[0,1]$},
\end{equation}
for some constant $C>0$ and for a constant $\eps_0>0$  such that \eqref{vanish-fourthorder} holds. 

To establish \eqref{tight-crit-applied}, fix $t,s\in[0,1]$, $s\leq t$ and write 
\begin{align}
\E \big| X_m(t)-X_m(s) \big|^4 &= \frac{r_m^4}{m^4}\sum_{i,j,k,\l} \E\big( \ol{h}(Y_i) \ol{h}(Y_j) \ol{h}(Y_k) \ol{h}(Y_\l) \big), \label{bigsum}
\end{align}
where we let $\ol{h}(x)=\ind{s< \Phi(x)\leq t}-(t-s)-(c_1(t)-c_1(s))x =h_t(x)-h_s(x)$.
Now, we split the sum in the RHS of \eqref{bigsum} following the value of  the cardinal of $\{i,j,k,\l\}$.

\paragraph{Sum over $\#\{i,j,k,\l\}= 1$}

The corresponding summation is $\frac{r_m^4}{m^4}\sum_{i=1}^m \E\left( (\ol{h}( Y_i))^4 \right)$. We have
\begin{align}
 \E\left( (\ol{h}( Y_i))^4 \right) \leq 3^4  \left(|t-s| + |t-s|^4 + \E(|Y_1|^4) L^4|t-s|^{4 /2} \right) \leq   C_1 |t-s|,\label{majnorm4},
 \end{align}
 for $C_1=5 \:3^4 L^4>0$.
Since $r_m^2\leq m$, we obtain
\begin{align}
 \frac{r_m^4}{m^4}\sum_{i=1}^m \E\left( (\ol{h}( Y_i))^4 \right) \leq \frac{C_1}{m} |t-s|.\label{majeasy}
 \end{align}

\paragraph{Sum over $\#\{i,j,k,\l\}= 2$}

Up to a multiplicative constant, we should consider the sum
\begin{align*}
 \frac{r_m^4}{m^4}\sum_{i\neq j} \E\left( (\ol{h}( Y_i))^2(\ol{h}( Y_j))^2 \right)=T_1^{(1)} + T_2^{(1)},
\end{align*} 
   where,  for an arbitrary $\eta_1>0$,  $T_1^{(1)}$ and $T_2^{(1)}$ are defined by   
   \begin{align}
   T_1^{(1)} &= \frac{r_m^4}{m^4}\sum_{i\neq j} \ind{|\Gamma_{i,j}|> \eta_1} \E\left( (\ol{h} ( Y_i))^2(\ol{h} ( Y_j))^2 \right)\label{equT1};\\
   T_2^{(1)}&=  \frac{r_m^4}{m^4}\sum_{i\neq j} \ind{|\Gamma_{i,j}|\leq \eta_1} \E\left( (\ol{h} ( Y_i))^2(\ol{h} ( Y_j))^2 \right)\label{equT2}.
   \end{align}
On the one hand, by using \eqref{majnorm4},
   \begin{align}
   T_1^{(1)} \leq  \frac{r_m^4}{\eta_1^2 m^4}\sum_{i\neq j}  |\Gamma_{i,j}|^2 
  \E\left( (\ol{h}( Y_i))^2(\ol{h}( Y_j))^2 \right)  \leq \frac{C_1}{\eta_1^2 m} \left(\frac{r_m^2}{ m^2}\sum_{i\neq j}  |\Gamma_{i,j}|^2 \right) |t-s|.
  \label{majT11}
  \end{align}
On the other hand,  by using \eqref{KSupperbound} in Proposition~\ref{KSprop} (with $g_1=g_2=(\ol{h})^2$ and $d=2$), we obtain that for any $i\neq j$ such that $|\Gamma_{i,j}|\leq \eta_1$ (choosing $\eta_1>0$ such that $2\sqrt{3\eta_1} <1$), 
  \begin{align*}
\E\left( (\ol{h}( Y_i))^2(\ol{h}( Y_j))^2 \right) \leq \frac{1}{(1-2\sqrt{3\eta_1})^2}\left(\E\left(|\ol{h}(Z)|^{8/3}\right)\right)^{3/2} \leq \frac{C_2}{(1-2\sqrt{3\eta_1})^2} |t-s|^{3/2},
  \end{align*}
for $C_2=3^4 L^4 \left(\E\left(|Z|^{8/3}\right)\right)^{3/2}\in(0,\infty)$
. Hence,  we get 
 \begin{align}
   T_2^{(1)} \leq  \frac{C_2}{(1-2\sqrt{3\eta_1})^2} |t-s|^{3/2}.\label{majT12}
   \end{align}

\paragraph{Sum over $\#\{i,j,k,\l\}= 3$}

Up to a multiplicative constant, we should consider the sum
\begin{align*}
 \frac{r_m^4}{m^4}\sum_{i, j, k \:\neq} \E\left( \ol{h}( Y_i) \ol{h}( Y_j) (\ol{h}( Y_k))^2 \right)=T_1^{(2)} + T_2^{(2)},
\end{align*}
   where,  for an arbitrary $\eta_2>0$,  $T_1^{(2)}$ and $T_2^{(2)}$ are defined similarly to \eqref{equT1} and \eqref{equT2}, by separating the case where $\max_{e_1\neq e_2\in \{i,j,k\}}{|\Gamma_{e_1,e_2}|}$ is above or below $\eta_2$. 

On the one hand, by using \eqref{majnorm4}, we have
\begin{align}
T_1^{(2)} &\leq  \frac{C_1 r_m^4}{m^4}\sum_{i, j, k \:\neq}  \ind{\max_{e_1\neq e_2\in \{i,j,k\}} {|\Gamma_{e_1,e_2}|}>\eta_2 } |t-s|\nonumber\\
&\leq 3\frac{C_1/\eta_2^4 }{m} \left(\frac{r_m^4}{m^2}\sum_{i\neq j}  |\Gamma_{i,j}|^4 \right)|t-s|. \label{majT21}
\end{align}
On the other hand,  by using \eqref{KSupperbound2} in Proposition~\ref{KSprop} (with $g_1=g_2=\ol{h}$, $g_3=(\ol{h})^2$, $f_1=f_2=\ind{s<\pnorm(\cdot)\leq t}$, $d=3$ and $d'=2$), we obtain that for any distinct $i, j, k$ such that $\max_{e_1\neq e_2\in \{i,j,k\}} {|\Gamma_{e_1,e_2}|}\leq \eta_2$ (choosing $\eta_2>0$ such that $3\sqrt{3\eta_2} <1$), 
 \begin{align*}
\E\left( \ol{h}( Y_i) \ol{h}( Y_j) (\ol{h}( Y_k))^2 \right) \leq \max_{e_1\neq e_2\in \{i,j,k\}} {|\Gamma_{e_1,e_2}|^2} \frac{27^{2}}{(1-3\sqrt{3\eta_2} )^3} |t-s|^{3/2} \times \sqrt{C_2}. 
  \end{align*}
  This yields
   \begin{align}
T_2^{(2)} &\leq   3\frac{ \sqrt{C_2} \:27^{2}}{(1-3\sqrt{3\eta_2} )^3}  \left(\frac{r_m^2}{m^2}\sum_{i \neq j}  |\Gamma_{i,j}|^2\right) |t-s|^{3/2}\label{majT22}.
  \end{align}

  \paragraph{Sum over $\#\{i,j,k,\l\}= 4$}

The last sum to be considered is 
\begin{align*}
 \frac{r_m^4}{m^4}\sum_{i, j, k,\l \:\neq} \E\left( \ol{h}( Y_i) \ol{h}( Y_j) \ol{h}( Y_k) \ol{h}( Y_\l)  \right)=T_1^{(3)} + T_2^{(3)},
\end{align*}
   where,  for an arbitrary $\eta_3>0$,  $T_1^{(3)}$ and $T_2^{(3)}$ are defined similarly to \eqref{equT1} and \eqref{equT2}, by separating the case where $\max_{e_1\neq e_2\in \{i,j,k,\l\}}{|\Gamma_{e_1,e_2}|}$ is above or below $\eta_3$. 
As before, 
\begin{align}
T_1^{(3)} &\leq  \frac{C_1 r_m^4}{m^4}\sum_{i, j, k,\l \:\neq}  \ind{\max_{e_1\neq e_2\in \{i,j,k,\l\}} {|\Gamma_{e_1,e_2}|}>\eta_3 } |t-s|\nonumber\\
&\leq 6 \frac{C_1}{\eta_3^4}  \left(\frac{r_m^4}{m^2}\sum_{i\neq j}  |\Gamma_{i,j}|^4 \right)|t-s|. \label{majT31}
\end{align}
Next, by using \eqref{KSupperbound2} in Proposition~\ref{KSprop} (with $g_i=\ol{h}$, $f_i=\ind{s<\pnorm(\cdot)\leq t}$ and  $d'=d=4$), we obtain that (choosing $\eta_3>0$ such that $4\sqrt{3\eta_3} <1$), 
 \begin{align}
T_2^{(3)} & \leq  \frac{r_m^4}{m^4}\sum_{i, j, k,\l \:\neq} \max_{e_1\neq e_2\in \{i,j,k,\l\}} {|\Gamma_{e_1,e_2}|^4} \frac{48^{4}}{(1-4\sqrt{3\eta_3})^4}  |t-s|^{3}\nonumber\\
&\leq  6\frac{48^{4}}{(1-4\sqrt{3\eta_3})^4} \left(\frac{r_m^4}{m^2}\sum_{i\neq j}  |\Gamma_{i,j}|^4 \right) |t-s|^{3} \label{majT32}.
  \end{align}

Finally, we obtain \eqref{tight-crit-applied} by combining the bounds \eqref{majeasy},\eqref{majT11},\eqref{majT12},\eqref{majT21},\eqref{majT22},\eqref{majT31},\eqref{majT32} and by using the assumptions \eqref{vanish-secondorder} and \eqref{vanish-fourthorder}.

 \subsection{Tightness under \eqref{vanish-secondorder+} and \eqref{gammam+}}\label{sec:autrecond}

Obviously, \eqref{vanish-secondorder+} and  \eqref{gammam+} imply \eqref{vanish-secondorder}, \eqref{equtheta} with $\theta=+\infty$, and $r_m\sim \gamma_m^{-1/2}$. Hence, 
Proposition~\ref{propGtilde} entails that the finite dimensional laws of $X_m=r_m(\Gtilde-I)$ converge to $0$ and
it only remains to prove that $X_m$ is tight. This can be done as in the previous section, except that we use $\kappa=2$ in Proposition~\ref{prop:tight}. Namely, we prove that,
for large $m$,
\begin{equation}\label{tight-crit-applied2}
\E \big| X_m(t)-X_m(s) \big|^2 \leq   C  \gamma_m^{\delta_0} |t-s| , \mbox{ for all $t,s\in[0,1]$},
\end{equation}
for some constants $C>0$, $\delta_0>0$. To prove \eqref{tight-crit-applied2},  we write (by using the same notation as in the previous section)
\begin{align*}
\E \big| X_m(t)-X_m(s) \big|^2=&\frac{r_m^2}{m^2} \sum_{i,j} \E\big( \ol{h}(Y_i) \ol{h}(Y_j) \big)\\
\leq& C_3\frac{r_m^2}{m} |t-s| +  C_3 |t-s|\frac{r_m^2}{\eta^2 m^2} \sum_{i\neq j} |\Gamma_{i,j}|^2 \\
&+ \frac{r_m^2}{m^2} \sum_{i\neq j} \ind{|\Gamma_{i,j}|\leq \eta} \E\big( \ol{h}(Y_i) \ol{h}(Y_j) \big)
\end{align*}
for some $\eta>0$ and by letting $C_3= 4\:3^2 L^2>0$.
Applying now \eqref{KSupperbound2} in Proposition~\ref{KSprop} (with $g_i=\ol{h}$, $f_i=\ind{s<\pnorm(\cdot)\leq t}$ for $i=1,2$ and  $d'=d=2$), we obtain that (choosing $\eta>0$ such that $2\sqrt{3\eta} <1$), 
$$\frac{r_m^2}{m^2} \sum_{i\neq j} \ind{|\Gamma_{i,j}|\leq \eta} \E\big( \ol{h}(Y_i) \ol{h}(Y_j) \big) \leq  \frac{(12)^{2}}{(1-2\sqrt{3\eta})^2} |t-s|^{3/2} \left(\frac{r_m^2}{m^2} \sum_{i\neq j} |\Gamma_{i,j}|^2 \right).$$
Finally, since \eqref{vanish-secondorder+} and  \eqref{gammam+} provide $\frac{r_m^2}{m^2} \sum_{i\neq j} |\Gamma_{i,j}|^2 =O( \gamma_m^{\eps_1})$ and $r_m^2/m=O(\gamma_m^{\eps_2})$ for some $\eps_1,\eps_2>0$, the criterion \eqref{tight-crit-applied2} is proved with $\delta_0=\eps_1 \wedge \eps_2$ and the proof is finished.

\appendix

\section{Technical results for proving the main theorem}   \label{sec:keylem}

\begin{lemma}\label{lemmacompens} 
Assume that $\Gamma^{(m)}$ satisfies  \eqref{vanish-secondorder} and \eqref{eigenvalues}. 
For $1\leq i \leq m$, let us consider the filtration $\{\mathcal{F}_i\}_{0\leq i \leq m}$ defined by $\mtc{F}_0=\sigma(\emptyset)$ and $\mathcal{F}_i=\sigma(Y_1,\dots,Y_i)$, and denote $\sigma_i^2=\var \left[\E \left( Y_i \cond \mathcal{F}_{i-1}\right) \right]$.
Consider the function $h_t(\cdot)$ defined by \eqref{defht}, the Hermite polynomials $H_\l(\cdot)$ defined by \eqref{def:hermite} and the coordinates $c_\l(\cdot)$ defined by \eqref{expan:clt}. 
Then  the following holds:
\begin{align}
&   \frac{r_m^2}{m^2} \sum_{i=1}^m \sigma_i^2  \rightarrow 0\label{conv0} ;\\
&\frac{r_m^2}{m^2} \sum_{i,j}  \left( \E \left[ \E \left( Y_i \cond \mathcal{F}_{i-1}\right) \E \left( Y_j \cond \mathcal{F}_{j-1}\right) \right]\right)^2
 \rightarrow 0;\label{conv0bis}\\
&  \frac{r^2_m}{m^2} \sum_{i=1}^m \E \left[\left( \E \left( h_t(Y_i) \cond \mathcal{F}_{i-1}\right) \right)^2 \right]\rightarrow 0, \:\:\mbox{ for any $t \in[0,1]$}\label{conv1} ;\\
&  \E \left[\left( \frac{r_m}{m} \sum_{i=1}^m \E \left( h_t(Y_i) \cond \mathcal{F}_{i-1}\right) \right)^2 \right]\rightarrow 0, \:\:\mbox{ for any $t \in[0,1]$}\label{conv2} .
\end{align}
\end{lemma}

\begin{proof}
By using Cholesky's decomposition, we can write $\Gamma=RR^T$ where $R$ is $m\times m$ a lower triangular matrix. Hence, denoting by  $R_{1,.},\dots R_{m,.}$ the lines of $R$, we have $<R_{i,.},R_{j,.}>=\Gamma_{i,j}$ for all $i,j$. Moreover, since we can write $Y_i=\sum_{j=1}^i R_{i,j} Z_{j}$ for some $Z_1,\dots,Z_m$ i.i.d. $\mathcal{N}(0,1)$, we have  $R_{i,i}^2=\var(Y_i \cond \mathcal{F}_{i-1})=1-\sigma_i^2$ and $\sigma_i^2=\sum_{j=1}^{i-1} R_{i,j}^2 $ for all $i$. 

Let us now prove \eqref{conv0}. From \eqref{eigenvalues}, we have for all $x\in\R^m$, $||R^T x||^2=x^T \Gamma x\geq \eta ||x||^2 $. Hence for all $x\in\R^m$, $||R x||^2\geq \eta ||x||^2 $. Thus, we have
\begin{align*}
\sum_{i<j} \Gamma_{i,j}^2 &=  \sum_{j=1}^m \sum_{i=1}^{j-1} \left( [R R_{j.}^T]_i \right)^2 \geq \eta \sum_{j=1}^m \sum_{i=1}^{j-1} R_{j,i}^2 = \eta\sum_{i=1}^m \sigma_i^2,
\end{align*}
which proves  \eqref{conv0} by  \eqref{vanish-secondorder}. As for \eqref{conv0bis}, we have for $i< j$,
\begin{align*}
\E \left[ \E \left( Y_i \cond \mathcal{F}_{i-1}\right) \E \left( Y_j \cond \mathcal{F}_{j-1}\right) \right] &=\E\left[ \sum_{k=1}^{i-1} R_{i,k} Z_{k} \sum_{\l=1}^{j-1} R_{j,\l} Z_{\l} \right] =  \sum_{k=1}^{i-1} R_{i,k} R_{j,k} =\Gamma_{i,j} - R_{i,i} R_{j,i}.
\end{align*}
Hence, we obtain
\begin{align*}
&\sum_{i<j}  \left( \E \left[ \E \left( Y_i \cond \mathcal{F}_{i-1}\right) \E \left( Y_j \cond \mathcal{F}_{j-1}\right) \right]\right)^2=\sum_{i<j}  (\Gamma_{i,j} - R_{i,i} R_{j,i})^2 \leq 2 \left(\sum_{i<j}  \Gamma_{i,j}^2+ \sum_{i<j} R_{j,i}^2 \right),
\end{align*}
which establishes \eqref{conv0bis} by \eqref{conv0} and \eqref{vanish-secondorder}.

Next, let us establish the following equality in $L^2(\P_m)$: for any $i=1,\dots,m$ and $t\in[0,1]$,
\begin{align}\E \left( h_t(Y_i) \cond \mathcal{F}_{i-1}\right) = \sum_{\l\geq 2} \frac{c_\l(t)}{\l !}  \sigma_i^{\l} H_\l \left(\frac{\E \left( Y_i \cond \mathcal{F}_{i-1}\right)}{\sigma_i}\right), \label{decomp:htYicond}
\end{align}
where the RHS of \eqref{decomp:htYicond} is $0$ if $\sigma_i=0$. 
For this,  consider some $1\leq i \leq m$ and assume $\sigma_i>0$ (otherwise the result is obvious). Let $\wt{Y}_i=\frac{\E \left( Y_i \cond \mathcal{F}_{i-1}\right)}{\sigma_i} \sim \mtc{N}(0,1)$. By using the multivariate Gaussian structure of $Y$, the distribution of $Y_i$ conditionally on $\mathcal{F}_{i-1}$ only depends on $\wt{Y}_i$. Hence, we can write $\E \left( h_t(Y_i) \cond \mathcal{F}_{i-1}\right)=g(\wt{Y}_i )$ for a (unique) function $g$ in $L^2(\R,\mtc{N}(0,1))$. 
We now consider the expansion of $g$ w.r.t. the Hermite polynomials in that space:  
$$
g(\cdot) = \sum_{\l\geq 0} \frac{\E(g(\wt{Y}_i) H_\l(\wt{Y}_i))}{\l!} H_\l(\cdot),
$$
and we can compute each coordinate $\E(g(\wt{Y}_i) H_\l(\wt{Y}_i))$ in the following way:
for any $\l\geq 0$,
\begin{align*}
\E \left[H_\l (\wt{Y}_i) \E \left( h_t(Y_i) \cond \mathcal{F}_{i-1}\right)\right] &=  \E \left[H_\l (\wt{Y}_i)  h_t(Y_i)\right]\\
&= \sum_{\l' \geq 2} \frac{c_{\l'}(t)}{(\l') !} \E \left[H_\l (\wt{Y}_i)  H_{\l'}(Y_i)\right]\\
& =\frac{c_{\l}(t)}{\l !} \sigma_i^\l \: \l !\:\ind{\l\geq 2}, 
\end{align*}
by using  Fubini's theorem (because $ \sum_{\l' \geq 2} \frac{|c_{\l'}(t)|}{(\l') !} \E \left[|H_\l (\wt{Y}_i)  H_{\l'}(Y_i)|\right]\leq (\l!)^{1/2} \sum_{\l' \geq 2} \frac{|c_{\l'}(t)|}{(\l'!)^{1/2}}  <\infty$), and by applying \eqref{prop:hermite} with $\cov(Y_i, \wt{Y}_i)= \sigma_i$. This proves \eqref{decomp:htYicond}.

Finally, by using \eqref{decomp:htYicond}, \eqref{prop:hermite} and notation above, we have
\begin{align*}
  \E \left[\left( \frac{r_m}{m} \sum_{i=1}^m \E \left( h_t(Y_i) \cond \mathcal{F}_{i-1}\right) \right)^2 \right]
  &= \frac{r_m^2}{m^2} \sum_{i,j}  \E \left[ \E \left( h_t(Y_i) \cond \mathcal{F}_{i-1}\right) \E \left( h_t(Y_j) \cond \mathcal{F}_{j-1}\right) \right]\\
  &= \frac{r_m^2}{m^2} \sum_{i,j} \sum_{\l\geq 2}\sum_{\l'\geq 2} \frac{c_\l(t)}{\l !}\frac{c_{\l'}(t)}{(\l' !)}  \sigma_i^{\l}\sigma_j^{\l'} \E \left[H_\l (\wt{Y}_i)  H_{\l'}(\wt{Y}_j)\right]\\
  &=\frac{r_m^2}{m^2} \sum_{i,j} \sum_{\l\geq 2} \frac{c_\l(t)^2}{\l !}  \sigma_i^{\l}\sigma_j^{\l} \left(\E \left[\wt{Y}_i \wt{Y}_j\right]\right)^\l\\
  &\leq \left(\sum_{\l\geq 2} \frac{c_\l(t)^2}{\l !}\right) \left(\frac{r_m^2}{m^2} \sum_{i,j}  \left(\E \left[ \E \left( Y_i \cond \mathcal{F}_{i-1}\right) \E \left( Y_j \cond \mathcal{F}_{j-1}\right) \right] \right)^2\right),
      \end{align*}
which proves \eqref{conv2} by using \eqref{conv0bis}. Exactly the same calculation with ``$i=j$" shows \eqref{conv1} from \eqref{conv0}.
\end{proof}

\begin{lemma}\label{lemmalambda} Assume that $\Gamma^{(m)}$ satisfies  \eqref{vanish-secondorder} and that $r_m^2 \var(\ol{Y}_m)$ converges to some positive real number. Consider the $(m+1)\times(m+1)$ covariance matrix $\Lambda^{(m+1)}$ of $(Y_i)_{0\leq i \leq m}$ defined in Section~\ref{sec:prooffinitedimforGchap}. 
Then the rate 
$$r_{m+1}(\Lambda^{(m+1)})= \left((m+1)^{-1} + \left| (m+1)^{-2} \sum_{0\leq i\neq j\leq m} \Lambda_{i,j}^{(m+1)}\right| \right)^{-1/2}$$
satisfies $r_{m+1}(\Lambda^{(m+1)})\sim r_m$ and moreover
\begin{align}\label{toprovelemmalambda}
(m+1)^{-2} r_m^2  \sum_{0\leq i\neq j \leq m} \left(\Lambda^{(m+1)}_{i,j}\right)^2   = o(1).
\end{align}
In particular, $\Lambda^{(m+1)}$ satisfies \eqref{vanish-secondorder}. Finally, when \eqref{equtheta} holds for $\Gamma^{(m)}$,  it also holds for $\Lambda^{(m+1)}$, with the same value of $\theta$.
\end{lemma}

\begin{proof}
By definition,
\begin{align*}
 m^{-2} \sum_{0\leq i\neq j\leq m} \Lambda_{i,j} &= m^{-2}  \sum_{1\leq i\neq j \leq m} \Gamma_{i,j} + 2 m^{-2}\sum_{1\leq j \leq m} \Lambda_{0,j}.\end{align*}
 Since $\Lambda_{0,j}=(\var \ol{Y}_m)^{-1/2} m^{-1} \sum_{i=1}^m \Gamma_{i,j}$, we have
\begin{align*}
 m^{-2}\sum_{1\leq j \leq m} \Lambda_{0,j} &= (\var \ol{Y}_m)^{-1/2} m^{-1} m^{-2} \sum_{1\leq i,j\leq m} \Gamma_{i,j}= m^{-1} \left(m^{-2} \sum_{1\leq i,j\leq m} \Gamma_{i,j}\right)^{1/2},
 \end{align*}
which is $o(1/m)$ because  $\Gamma$ satisfies  \eqref{vanish-secondorder} and thus \eqref{weakdep}. This implies $r_{m+1}(\Lambda)\sim r_m$.
Next, we establish \eqref{toprovelemmalambda}. Let us write
\begin{align*}
(m+1)^{-2} r_m^2  \sum_{0\leq i\neq j \leq m} \left(\Lambda_{i,j}\right)^2   = (m+1)^{-2} r_m^2 \left( \sum_{1\leq i\neq j \leq m} \left(\Gamma_{i,j}\right)^2 +  2 \sum_{1\leq j \leq m} \left(\Lambda_{0,j}\right)^2\right).
\end{align*}
Furthermore, we have
\begin{align*}
 \sum_{1\leq j \leq m} \left(\Lambda_{0,j}\right)^2 &= (\var \ol{Y}_m)^{-1} \sum_{1\leq j \leq m} \left( m^{-1} \sum_{i=1}^m \Gamma_{i,j}\right)^2 \\
 &\leq (\var \ol{Y}_m)^{-1} m^{-2} \sum_{1\leq i,i' \leq m} \left( 2 \Gamma_{i,i'} 
 + \sum_{j\notin\{i,i'\}} \Gamma_{i,j}\Gamma_{i',j}\right)\\
 &\leq 2 + (m\var \ol{Y}_m)^{-1}\sum_{1\leq i\neq j \leq m}  \left( \Gamma_{i,j}\right)^2.
 \end{align*}
This implies the result, because $m \var \ol{Y}_m \geq r_m^2  \var \ol{Y}_m$, which is bounded away from $0$ by assumption.
\end{proof}

\section{Results related to Hermite polynomials}\label{sec:Hermite}

  Let us first recall that the sequence of Hermite polynomials $H_\l(x)$, $\l\geq 0$, $x\in \R$, is defined by the expression: for all $\l\geq 0$, 
 \begin{equation}\label{def:hermite}
\forall x\in\R, \:\phi^{(\l)} (x) = (-1)^\l H_\l(x) \phi(x),
 \end{equation}
 where $\phi(x)=(2\pi)^{-1/2} \exp(-x^2/2)$ is the density of a Gaussian standard variable and $\phi^{(\l)}$ denotes its $\l$-th derivative (by convention, $\phi^{(0)}=\phi$). For instance, we have $H_0(x)=1$, $H_1(x)=x$ and $H_2(x)=x^2-1$. 
 
 A well known fact is that $\{H_\l(\cdot)/(\l !)^{1/2}, \l\geq 0\}$  is an Hilbert basis in $L^2(\R,\mtc{N}(0,1))$, the Hilbert space 
composed by 
square integrable functions w.r.t. the  standard Gaussian measure. Moreover, the following property holds: for any centered $2$-dimensional Gaussian vector $(U,V)$ with $\E U^2=\E V^2=1$,
 \begin{equation}\label{prop:hermite}
\forall \l,\l' \geq 0, \l\neq\l', \:\:\E (H_\l(U)H_{\l'}(V) ) = (\cov(U,V))^\l \:\l ! \:\delta_{\l,\l'}.
 \end{equation}
 The latter can be seen as a consequence of Mehler's formula, itself being nicely presented in \cite{Foa1981} (1.4) (see also references therein).

\paragraph{Proof of Proposition~\ref{prop:covG}}
  
Let us start by expanding, for any $t\in[0,1]$, the function $\ind{\Phi(\cdot)\leq t}$ w.r.t. the Hermite polynomial basis in $L^2(\R,\mtc{N}(0,1))$: 
 \begin{equation}\label{expan:indx}
\ind{\Phi(\cdot)\leq t}  =  \sum_{\l\geq 0} c_\l(t) H_\l(\cdot) / (\l !).
 \end{equation}
 By applying \eqref{expan:indx} at $Y_i$, we obtain the following expansion in $L^2(\P_m)$: for all $i=1,\dots,m$,
 \begin{equation}\label{expan:indYi}
\ind{\Phi(Y_i)\leq t}  =  \sum_{\l\geq 0} c_\l(t) H_\l(Y_i) / (\l !).
 \end{equation}
By averaging w.r.t. $i$, we obtain 
\begin{align}\label{expan:Gchap}
\G(t) - t = \sum_{\l\geq 1} \frac{c_\l(t)}{\l !} m^{-1} \sum_{i=1}^m H_\l(Y_i) .
\end{align} 
where the series in the RHS of \eqref{expan:Gchap} converges in $L^2(\P_m)$ (by using the triangle inequality).
The proof is finished  by combining \eqref{expan:Gchap} with \eqref{prop:hermite}.\\

Next, the following proposition shares some similarities with Lemma~4.5 of \cite{Taq1977} and Lemma~3 of \cite{CM1996}.
\begin{proposition}\label{KSprop} 
Consider an integer $d\geq 2$,  a positive number $\rho$ such that $\sqrt{3\rho} d<1$ and $Z\sim\mathcal{N}(0,1)$. Let $g_1,\dots,g_d$ be $d$ measurable real functions defined on $\R$ such that $\E\left(|g_i(Z)|^{4/3}\right)<+\infty$, $1\leq i \leq d$. 
Let $(U_1,\dots,U_d)$ be $d$-dimensional centered Gaussian vector with $\E U_i^2=1$, $1\leq i \leq  d$, and $|\E(U_i U_j)|\leq \rho$, $1\leq i \neq j \leq  d$. Then the following holds:
\begin{align}
\E\left[\prod_{i=1}^d \left|g_i(U_i)\right|\right] &\leq  \frac{1}{(1-\sqrt{3\rho} d)^d} \prod_{i=1}^d \left(\E\left(|g_i(Z)|^{4/3}\right)\right)^{3/4}\label{KSupperbound};
\end{align}
Furthermore, if $\E(g_i(Z))=0$  and  $\E(Zg_i(Z))=0$ for $1\leq i \leq d'$ for an integer $d'$, $1\leq d'\leq d$, we have 
\begin{align}
\left|\E\left[\prod_{i=1}^d g_i(U_i)\right] \right|&\leq \rho^{d'}  \frac{(3d^2)^{d'}}{(1-\sqrt{3\rho} d)^d} \prod_{i=1}^d \left(\E\left(|f_i(Z)|^{4/3}\right)\right)^{3/4}\label{KSupperbound2},
\end{align}
where $f_i$ is any function such that $f_i(x)=g_i(x)-\alpha_i-\beta_i x$, $x\in\R$, $\alpha_i,\beta_i\in \R$, for $1\leq i \leq d'$ and $f_i=g_i$ otherwise.
\end{proposition}

\begin{proof}
The Kibble-Slepian formula \cite{Kib1945,Sle1972} (given, e.g., in expression (2.2) of \citealp{Foa1981}) provides that
\begin{align}
\E\left[\prod_{i=1}^d g_i(U_i)\right] &= \E\left(\sum_{\nu} \prod_{i<j} \frac{(\E(U_i U_j))^{\nu_{ij}}}{\nu_{ij} !} .  \prod_{i=1}^d g_i(Z) H_{\nu_{i.}}(Z)\right)\nonumber\\
&= \sum_{\nu} \prod_{i<j} \frac{(\E(U_i U_j))^{\nu_{ij}}}{\nu_{ij} !} .  \prod_{i=1}^d \E(g_i(Z) H_{\nu_{i.}}(Z)), \label{KSformula}
\end{align}
where the summation is over all the $d\times d$ symmetric matrix $\nu=(\nu_{ij})_{1\leq i,j\leq d}$ with nonnegative integral entries and with diagonal entries equal to zero, while $\nu_{i.}$ denotes $\nu_{i1}+\dots+\nu_{id}$.
Above, we have implicitly  used Fubini's theorem (the summation over $\nu$ is infinite). 
The next calculations show that this is indeed valid: by using the assumptions, we have 
\begin{align} 
&\sum_{\nu} \prod_{i<j} \frac{|\E(U_i U_j)|^{\nu_{ij}}}{\nu_{ij} !} .  \prod_{i=1}^d \E|g_i(Z) H_{\nu_{i.}}(Z)|\nonumber\\
&\leq  \sum_{\nu}  \prod_{i=1}^d \left(  \frac{ \rho^{\nu_{i.}}}{\prod_{j} \nu_{ij} !} \right)^{1/2} \E|g_i(Z) H_{\nu_{i.}}(Z)|\nonumber\\
&\leq  \sum_{x_1,\dots, x_d \in \mathbb{N}^d} \prod_{i=1}^d \left(  \frac{ \rho^{x_{i.}}}{\prod_{j} x_{ij} !} \right)^{1/2} \E|g_i(Z) H_{x_{i.}}(Z)|\nonumber\\
&=\prod_{i=1}^d\left[\sum_{y \in \mathbb{N}^d} \left(  \frac{ \rho^{y_{1}+\dots + y_d}}{\prod_{j} y_{j} !} \right)^{1/2} \E|g_i(Z) H_{y_{1}+\dots + y_d}(Z)| \right]\nonumber\\
&=\prod_{i=1}^d\left[ \sum_{\l\geq 0} \rho^{ \l/2} \E\left|g_i(Z) H_{\l}(Z) / (\l !)^{1/2}\right|\sum_{\substack{y \in \mathbb{N}^d\\ y_{1}+\dots + y_d=\l}} \left(  \frac{ \l!}{\prod_{j} y_{j} !} \right)^{1/2}    \right].\label{hotcomput}
\end{align}
Now, in the latter display, the sum over $y$ is upper bounded by $d^\l$, which gives that the RHS of \eqref{hotcomput} is upper bounded by
\begin{align*}
 \sum_{\l\geq 0} (\rho d^2)^{ \l/2}  \E\left|g_i(Z) H_{\l}(Z) / (\l !)^{1/2}\right|  
&\leq \left(\sum_{\l\geq 0} (3\rho d^2)^{ \l/2} \right) \left(\E\left(|g_i(Z)|^{4/3}\right)\right)^{3/4},
\end{align*}
where the latter combines H\"older's inequality with Lemma~\ref{lemnormherm} (used with $p=4$). 
This proves \eqref{KSupperbound} and shows that Fubini's theorem can be applied to get \eqref{KSformula}.

Finally, we prove \eqref{KSupperbound2} by using \eqref{KSformula} and the same calculations as above, except that the absolute values should be kept outside the expectations. As a result, for $1\leq i \leq d'$, since $\E(g_i(Z) H_{\l}(Z))=0$ for $\l=0,1$ by assumption, the corresponding sums over $\l$ start at $\l=2$. This establishes \eqref{KSupperbound2}, because for all $\l\geq 2$ and $1\leq i \leq d'$, $\E(g_i(Z) H_{\l}(Z))=\E(f_i(Z) H_{\l}(Z))$.
\end{proof}

The following result was obtained in the proof of Lemma~3.1 in \cite{Taq1977}. We provide an elementary proof below.
Also, let us mention that there are more accurate such results when $\l$ grows to infinity, see Theorem~2.1 in \cite{Lar2002}.

\begin{lemma}\label{lemnormherm}
For all even integer $p\geq 2$ and $\l \geq 0$, we have
$\left[\E \left(H_\l(Z)/\sqrt{\l !}\right)^{p}\right]^{1/p} \leq (p-1)^{\l/2},$ for $Z\sim\mathcal{N}(0,1).$
\end{lemma}

\begin{proof}
For some $\l\geq 1$, by using $H_\l'=\l H_{\l-1}$ and \eqref{def:hermite}, we obtain
\begin{align*}
\int [H_\l(x)]^p \phi(x)dx&=(-1)^\l \int [H_\l(x)]^{p-1} \phi^{(\l)}(x) \, dx,\\
&=\l(p-1) \int [H_\l(x)]^{p-2} [H_{\l-1}(x)]^2 \phi(x)\, dx.
\end{align*}
Next, by using H\"older's inequality, we get
$\left(\int [H_\l(x)]^p \phi(x)dx\right)^{2/p}\leq \l (p-1)
\left(\int |H_{\l-1}(x)|^p \phi(x)dx\right)^{2/p},$
and the result is obtained by induction on $\l$.
\end{proof}

\begin{lemma}\label{conseqMelher} 
Consider the function $h_t(\cdot)$ defined by \eqref{defht} and  $c_\l(\cdot)$ defined by \eqref{expan:clt}. 
Let us consider a two-dimensional centered Gaussian vector $(U,V)$ with $ \E U^2 = \E V^2=1$.
Then for any $t,s \in[0,1]$, the following holds:
\begin{align}
\E(h_t(U)h_s(V)) = \sum_{\l\geq 2} \frac{c_\l(t)c_\l(s)}{\l!} (\cov(U,V))^\l \label{equ-2fact}.
 \end{align}
\end{lemma}

\begin{proof}
Expression \eqref{equ-2fact} is a direct consequence of \eqref{prop:hermite} and of Fubini's theorem.
\end{proof}

\begin{lemma}\label{lem:majc1}
The function $c_1(\cdot)= \phi(\Phi^{-1}(\cdot))$ satisfies the following: for all $\nu\in(0,1)$, there exists some constant $C_\nu>0$ such that for all $s,t\in[0,1]$,
\begin{align}\label{majc1}
|c_1(t)-c_1(s)|\leq C_\nu |t-s|^{1-\nu}.
\end{align}
\end{lemma}

\begin{proof}
First note that the derivative of $c_1$ on $(0,1)$ is $\Phi^{-1}$.
Classically (see, e.g., Lemma~12.3 of \cite{ABDJ2006}), there is some $x_0\in(0,1/2)$ such that for any $u\in(0,x_0)$, $\Phi^{-1}(u)\leq \sqrt{2\log (1/u)}.$ Also, obviously, for some fixed $\nu>0$, there is some $C'_\nu>0$ such that for any $u\in(0,x_0)$, $\sqrt{2\log (1/u)}\leq C'_\nu u^{-\nu}$. 
As a consequence, since $|\Phi^{-1}|$ is bounded on $[x_0,1-x_0]$, there exists some constant $C''_\nu>0$ such that for all $u\in(0,1)$, $|\Phi^{-1}(u)|\leq C''_\nu u^{-\nu}$.
This entails that for all $0<s\leq t<1$,
$$
| c_1(t)-c_1(s) |\leq \int_s^t | \Phi^{-1} (u)| du \leq \frac{C''_\nu}{1-\nu} (t^{1-\nu} - s^{1-\nu}) \leq C_\nu (t-s)^{1-\nu}
$$
 by letting $C_\nu =C''_\nu/(1-\nu)>0$ and because $(x+y)^\delta\leq x^\delta+y^\delta$ for any $x,y\geq 0$ and any $\delta\in(0,1)$.
\end{proof}

\section{Useful auxiliary results} \label{sec:iff}

The following result can certainly be considered as well known, although we failed to find a precise reference for it. It can be seen as a reformulation in our framework of classical tightness results as given, e.g., in Lemma~2 of \cite{CM1996}, in Remark~2.1 of \cite{SY1996}  and Proposition 6 of \cite{DP2007}.

  \begin{proposition}[Tightness criterion for empirical distribution function with non-standard scaling parameters]\label{prop:tight}
Consider $\xi_1,\dots,\xi_m$ real random variables  (that need not to be independent or identically distributed) such that 
$\overline{\xi}_m\xrightarrow{P} 0$ as $m$ tends to infinity, for $\overline{\xi}_m=m^{-1}\sum_{i=1}^m \xi_m$, and consider the process
 $$Z_m(t)= (a_m/m) \sum_{i=1}^m g_t(\xi_i), \mbox{ for $t\in[0,1]$,}$$
where $(a_m)_m$ is some positive sequence tending to infinity as $m$ tends to infinity and where $g_t(x)=\ind{\pnorm(x)\leq t}-f_0(t)-f_1(t) x$ for functions $f_0,f_1$ on $[0,1]$ such that $|f_0(t)-f_0(s)| \vee |f_1(t)-f_1(s)| \leq L |t-s|^q$, $0\leq s,t\leq 1$,  for some $q\in(0,1]$ and $L>0$.
Assume that the following holds: for large $m$,
\begin{equation}\label{tight-crit}
\E \big| Z_m(t)-Z_m(s) \big|^\kappa
\leq   C\big( |t-s|^{\delta_1} + (a_m)^{-\delta_2/q} |t-s|^{q'} \big), \mbox{ for all $t,s\in[0,1]$},
\end{equation}
for constants $\kappa>0$, $C>0,$ $\delta_1>1$, $q'\in(0,1]$ and $\delta_2  >1-q'$.
Then, as $m$ grows to infinity, the sequence of processes  $(Z_m)_m$ is tight in $D(0,1)$ (endowed with the Skorokhod topology and the corresponding Borel $\sigma$-field) and any limit is a.s. a continuous process.
\end{proposition}

\begin{proof}
The proof 
is based on standard arguments and is similar to the proof of Theorem~22.1 in \cite{Bill1968}.
Fix $\varepsilon\in(0,1)$ and $\eta>0$. 
Following Theorem~15.5 in \cite{Bill1968}, it is sufficient to prove that  there exists a $\delta \in (0,1)$ such that for large $m$,
\begin{align*}
 \P \left( \sup_{\substack{0\leq s,t \leq 1\\ |s- t|\leq \delta}} | Z_m(t) - Z_m(s) | >\eps \right) < \eta.
\end{align*}
We merely check (see, e.g.,  the proof of Theorem~8.3 in \cite{Bill1968}) that the latter holds if there exists $\delta \in (0,1)$ such that for large $m$,
\begin{align}\label{equ-relcomp}
\forall s \in [0,1],\:\: \P \left( \sup_{t: s\leq t \leq (s+\delta)\wedge 1} | Z_m(t) - Z_m(s) | >\eps \right) < \eta \delta.
\end{align}
Let us now prove \eqref{equ-relcomp}. Fix $s \in [0,1]$. Assumption \eqref{tight-crit} entails that for all $u,v\in[0,1]$ such that $(v-u)^q \geq \eps/a_m$, we have 
$$
\E \big| Z_m(v)-Z_m(u) \big|^\kappa
\leq   \frac{2C}{\eps^{\delta_2/q}} |v-u|^{\delta_3}  
$$
for $\delta_3=\delta_1\wedge(q'+\delta_2)>1$.
Hence, if $p>0$ is such that $p^{q}\geq \eps/a_m$, applying Theorem~12.2 of \cite{Bill1968} we have  for all integer $M$ such that $s+Mp\leq 1$ and for all $\lambda>0$, 
\begin{align}\label{equ-lambda}
\P\left( \max_{1\leq i\leq M} | Z_m(s+i p) - Z_m(s) | >\lambda \right) \leq \frac{K}{\lambda^\kappa \eps^{\delta_2/q}} (Mp)^{\delta_3}
\end{align}
for some positive constant $K>0$ (only depending on $\delta_3$, $\kappa$ and $C$). 
Next, we use the following inequality: for all $0\leq u,v\leq 1$, $u \leq v \leq u+p$,
\begin{align}\label{equ-majecdf}
| Z_m(v) - Z_m(u) |\leq | Z_m(u+ p) - Z_m(u) | + 2 L a_m p^q (1+|\ol{\xi}_m|).
\end{align}
The latter holds because we have 
\begin{align*}Z_m(v) - Z_m(u) &= (a_m/m) \sum_{i=1}^m  \ind{u < \xi_i\leq v }  - a_m (f_0(v)-f_0(u)) - a_m (f_1(v)-f_1(u)) \ol{\xi}_m \\
&\leq Z_m(u+p) - Z_m(u) + 2 L a_m p^q (1+|\ol{\xi}_m|)
\end{align*}
and 
$Z_m(u) - Z_m(v)\leq   a_m (f_0(v)-f_0(u)) + a_m (f_1(v)-f_1(u)) \ol{\xi}_m \leq  L a_m p^q (1+|\ol{\xi}_m|)$. 

Now, by using \eqref{equ-majecdf}, we obtain
\begin{align}\label{equ-lambda2}
\sup_{t: s\leq t \leq s+Mp} | Z_m(t) - Z_m(s) | \leq  3 \max_{1\leq i\leq M} | Z_m(s+ i p) - Z_m(s) | + 2 L a_m p^q  (1+|\ol{\xi}_m|).
\end{align}
Furthermore, provided that $a_m p^q \leq 2\eps$, we have $\P(2 L a_m p^q  (1+|\ol{\xi}_m|)>5L\eps) \leq \P(|\ol{\xi}_m|>1/4)$. 
Hence, combining \eqref{equ-lambda} and \eqref{equ-lambda2}, by taking $\delta\in(0,1)$ such that 
$K\delta^{\delta_3-1}/\eps^{\kappa+\delta_2/q}<\eta/2$, 
we will obtain that for all $s\in[0,1-\delta]$, for large $m$,
$$
\P \left( \sup_{t: s\leq t \leq s+\delta} | Z_m(t) - Z_m(s) | > (3+5L)\eps \right)\leq \frac{K}{ \eps^{\kappa+\delta_2/q}} \delta^{\delta_3} +
\P(|\ol{\xi}_m|>1/4)
<\eta \delta,
$$
as soon as we can choose $p>0$ and an integer $M$ such that $Mp=\delta$ and $\eps/a_m\leq p^q \leq 2\eps/a_m$. This holds if there exists an integer into the interval $[\delta (a_m/\eps)^{1/q},\delta (a_m/(2\eps))^{1/q}]$, which is true for large $m$ because $a_m$ tends to infinity. This entails \eqref{equ-relcomp} with $\eps$ replaced by $ (3+5L)\eps$ and the proof is finished.
\end{proof}

\begin{proposition}[Partial functional delta method  on $D(0,1)$]\label{prop:FDM}
Consider the linear space $D(0,1)$ of càd-làg function on $[0,1]$ and the linear space $C(0,1)$ of continuous functions on $[0,1]$.  Let $\theta=(\theta_0,\theta_1)\in D(0,1)^2$.
Let $\phi : D(0,1)^2 \mapsto \R$ be Hadamard differentiable at $\theta$ tangentially to $C(0,1)$, w.r.t. the supremum norm, and such that the derivative is of the form $$\dot{\phi}_\theta(H_0,H_1)= g_\theta(H_0), \mbox{ for any $(H_0,H_1)\in C(0,1)^2$},$$ for a continuous linear mapping $g_\theta: C(0,1) \mapsto \R$.
Consider $\mathbb{Z}_{0,m}$, $\mathbb{Z}_{1,m}$, $m\geq 1$, processes valued in $D(0,1)$ and $\mathbb{Z}_{0}$, $\mathbb{Z}_{1}$ two processes valued a.s. in $C(0,1)$. Assume that the two following distribution convergences hold (w.r.t. the Skorokhod topology and the corresponding Borel $\sigma$-field), for some positive sequence $(a_m)_m$ tending to infinity:
\begin{align*}
&a_m (\mathbb{Z}_{0,m}- \theta_0) \leadsto \mathbb{Z}_{0};\\
&a_m (\mathbb{Z}_{1,m}- \theta_1) \leadsto \mathbb{Z}_{1}.
\end{align*}
Then we have 
\begin{align}\label{concl:FDM}
a_m (\phi(\mathbb{Z}_{0,m},\mathbb{Z}_{1,m}) -\phi( \theta)) \leadsto g_\theta(\mathbb{Z}_{0}).
\end{align}
\end{proposition}

\begin{proof} 
Classically, let us show that for any subsequence $\{n\}$ there exists a further subsequence $\{\l\}$ such that \eqref{concl:FDM} holds along this subsequence. For any $\{n\}$, since both processes $a_{n} (\mathbb{Z}_{0,n}- \theta_0)$ and $a_{n} (\mathbb{Z}_{1,n}- \theta_1)$ are (Skorokhod-)tight, the joint process $\left(a_{n} (\mathbb{Z}_{0,n}- \theta_0),a_{n} (\mathbb{Z}_{1,n}- \theta_1)\right)$ also is.
Hence, by Prohorov's theorem, there exists a further subsequence $\{\l\}$ such that $\left(a_{\l} (\mathbb{Z}_{0,\l}- \theta_0),a_{\l} (\mathbb{Z}_{1,\l}- \theta_1)\right)$ converges in distribution. Now applying the Skorokhod's representation theorem (see, e.g., Theorem~6.7 page 70 in \cite{Bill1999}), there exists random elements $T_{\l}=(T_{0,\l},T_{1,\l})$, $\l\geq 1$, $T=(T_0,T_1)$, defined on a common probability space, such that $\mathcal{L}(T_{\l})=\mathcal{L}\left(a_{\l} (\mathbb{Z}_{0,\l}- \theta_0),a_{\l} (\mathbb{Z}_{1,\l}- \theta_1)\right)$, $\mathcal{L}(T_0)=\mathcal{L}(\mathbb{Z}_{0})$, $\mathcal{L}(T_1)=\mathcal{L}(\mathbb{Z}_{1})$ and $T_{\l}$ converges a.s. to $T$. Since both $T_0$ and $T_1$  belong  to $C(0,1)$ (a.s.) and since any sequence of càd-làg functions converging (w.r.t. to the Skorokhod distance) to a continuous function also converges uniformly, we obtain 
\begin{align*}
||T_{0,\l} - T_ 0 ||_\infty  + ||T_{1,\l} - T_1||_\infty \rightarrow 0 \:\:\:\mbox{ a.s.}   
\end{align*}
Hence, the Hadamard differentiability of $\phi$ entails:
\begin{equation*}
 \frac{\phi(\theta+t_\l T_\l) - \phi(\theta)}{t_\l} \rightarrow g_\theta(T_0) \:\:\:\mbox{ a.s.}   \:\:\:,
\end{equation*}
for any sequence $t_\l\rightarrow 0$. By taking $t_\l=1/a_\l$, we derive \eqref{concl:FDM} along the subsequence $\{\l\}$, which proves the result.
\end{proof}

\begin{lemma}\label{LLN}
Assume that $\Gamma$ satisfies \eqref{weakdep}. Then for any $h:\R \rightarrow \R$ measurable  such that $\E |h(Z)|<\infty$, we have
\begin{align}\label{prop:LLN}
m^{-1} \sum_{i=1}^m h(Y_i) \xrightarrow{P} \E [h(Z)],\:\: \mbox{ for $Z\sim\mtc{N}(0,1)$.}
\end{align}
\end{lemma}

\begin{proof}
By Section~\ref{sec:convcov}, Assumption \eqref{weakdep} implies that $\forall t \in [0,1],$ $\G(t)  \xrightarrow{P}  t$. 
Since $h\in L^1(\R, \mathcal{N}(0,1))$, for any $\eps>0$, there is a continuous bounded function $h_\eps$ such that $\E |h(Z)-h_\eps(Z)|\leq \eps$. Moreover, by definition of the weak convergence, \eqref{prop:LLN} holds for $h=h_\eps$ (for instance, the convergence in probability can be seen as an a.s. convergence up to consider subsequence). Since we have 
$$
\sup_{m\geq 1}\left\{ \E \left| m^{-1} \sum_{i=1}^m (h(Y_i) - h_\eps(Y_i))\right|\right\} \leq \sup_{m\geq 1}\left\{  m^{-1} \sum_{i=1}^m \E \left|h(Y_i) - h_\eps(Y_i)\right|\right\} \leq \eps,
$$
we can conclude by using Lemma~\ref{lemapproxeps}.
\end{proof}

The following lemma is classical, see, e.g., Theorem~4.2 in \cite{Bill1968}.
\begin{lemma}\label{lemapproxeps}
For $n\geq 1$ and $\eps>0$, let $X_n^\eps$, $X_n$, $X^\eps$, $X$ be real random variables ($X_n$ and $X_n^\eps$ being defined on the same probability space) and such that
\begin{itemize}
 \item[(a)] $\forall \eps>0$, $X_n^\eps \leadsto X^\eps$ as $n \to \infty$;
 \item[(b)] $X^\eps \leadsto X$ as $\eps \to 0$; 
 \item[(c)] $\limsup_{n\to \infty} \{\E |X_n^\eps-X_n |\} \rightarrow 0$ as $\eps \to 0$.
 \end{itemize}
 Then $X_n \leadsto X$.
\end{lemma}

\section*{Acknowledgments}
We are grateful to Jérôme Dedecker, Stephane Gaiffas, Pierre Neuvial and Mathieu Rosenbaum for helpful discussions.
The second author was supported by the French Agence Nationale de la Recherche (ANR grant references: ANR-09-JCJC-0027-01, ANR-PARCIMONIE, ANR-09-JCJC-0101-01) 
and by the French ministry of foreign and european affairs (EGIDE - PROCOPE project number 21887 NJ).

\bibliographystyle{apalike} 
\bibliography{biblio}
\end{document}